\documentclass[10pt]{article}

\usepackage{amsfonts, epsfig, amsmath, amssymb, color,mathabx,amsthm}
\usepackage{mathrsfs}
\usepackage[notref, notcite]{showkeys}

\usepackage[round]{natbib}   % omit 'round' option if you prefer square brackets
\usepackage{showkeys}
\usepackage[english]{babel}

\usepackage{color}

\newcommand{\prt}{\partial}

\newcommand{\E}{\mathbf{E}}
\renewcommand{\P}{\mathbf{P}}

\renewcommand {\epsilon}{\varepsilon}

\theoremstyle{plain}
\newtheorem{thm}{Theorem}[section]
\newtheorem{lem}[thm]{Lemma}
\newtheorem{prp}[thm]{Proposition}

\theoremstyle{definition}
\newtheorem{rem}[thm]{Remark}

\DeclareMathSymbol{\ophi}{\mathalpha}{letters}{"1E}

\renewcommand{\phi}{\varphi}

\newcommand{\be}{\begin{equation}}
\newcommand{\ee}{\end{equation}}
\newcommand{\ben}{\begin{equation*}}
\newcommand{\een}{\end{equation*}}

\newcommand{\ba}{\begin{equation}\begin{aligned}}
\newcommand{\ea}{\end{aligned}\end{equation}}
\newcommand{\ban}{\begin{equation*}\begin{aligned}}
\newcommand{\ean}{\end{aligned}\end{equation*}}

\DeclareMathOperator{\sgn}{sgn}

\newcommand{\ex}{\mathrm{e}}
\newcommand{\di}{\mathrm{d}}

\newcommand{\rF}{\mathscr{F}}

\newcommand{\bI}{\mathbb{I}}

\newcommand{\bR}{\mathbb{R}}

\let\oldmarginpar\marginpar
\renewcommand{\marginpar}[1]{\oldmarginpar{\scriptsize\texttt{\color{red}{#1}}}}

\allowdisplaybreaks[4]

\graphicspath {{pics/}}

\begin{document}
%\layout
\title{First order convergence of weak Wong--Zakai approximations of L\'evy driven Marcus SDEs}

% \date{\today\footnote{{\jobname}.tex} }

% \date{\today\footnote{{\jobname}.tex\hfill \textbf{Preliminary version!!! Do not distribute!!!}}}

\author{Tetyana Kosenkova\footnote{Institute of Mathematics,
University of Potsdam, Karl--Liebknecht--Strasse 24--25, 14476 Potsdam, Germany;
kosenkova@math.uni-potsdam.de},
Alexei Kulik,\footnote{Wroclaw University of Science and Technology; kulik.alex.m@gmail.com} \footnote{A.Kulik was supported in part by the National Science Centre, Poland, grant no.
2019/33/B/ST1/02923 and by Alexander
von Humboldt Foundation within the Research Group Linkage Programme between the
Institute of Mathematics at the University of Potsdam and the Institute of Mathematics
of National Academy of Sciences of Ukraine}\ \ and
Ilya Pavlyukevich\footnote{Institute of Mathematics, Friedrich Schiller University Jena, Ernst--Abbe--Platz 2,
07743 Jena, Germany; ilya.pavlyukevich@uni-jena.de}}

\maketitle

\begin{abstract}
% As soon as he saw the Big Boots, Pooh knew that an Adventure was going to happen...
For solutions $X=(X_t)_{t\in[0,T]}$ of L\'evy-driven Marcus stochastic differential equations we study
the Wong--Zakai type time discrete approximations $\bar X=(\bar X_{kh})_{0\leq k\leq T/h}$, $h>0$, and establish the first order
convergence $|\E f(X_T)-\E f(X^h_T)|\leq C h$ for $f\in C_b^4$.

\end{abstract}


\numberwithin{equation}{section}

\tableofcontents

\section{Introduction}

SDEe driven by L\'evy processes belong nowadays to a standard toolbox of researches working in Physics, Finance, Engineering etc.
Under standard assumptions, solutions $X$ of SDEs are Markov (Feller) processes containing the continuous diffusive component as well as
(infinitely many) jumps which model instant change of the observable in the phase space.

From the point of view of applications, one often wants to determine the averaged quantities of the type $\E_x f(X_T)$ for a fixed
deterministic time $T>0$ and a regular test function $f$ (which is equivalent to solving a certain integro-differential Poisson equation). This
is actually the question of an effective approximation in the Monte-Carlo methods, which constitute a natural analogue for finite elements method for PIDE.

The approximation problem for the functionals $\E_x f(X_T)$ for diffusions is nowadays a classical topic (\cite{KloPla-95}).
The numerical methods have originated in the paper by \cite{maruyama1955continuous} who showed that
for the
It\^o SDE $\di X=a(X)\di t+ b(X)\,\di W$ driven by the Brownian motion
the Euler scheme $\bar X_{(k+1)h}=\bar X_{kh}+ a(\bar X_{kh})h +b(\bar X_{kh})(W_{(k+1)h}-W_{kh})$
with the step size $h>0$
converges to $X_T$ in $L^2$-sense
for each $T\geq 0$.
\cite{mil1979method,talay1984efficient} showed that the Euler scheme yields weak convergence of the order $\mathcal O(h)$.
Higher order methods can be found
in \cite{mackevicius1994second,talay1984efficient,mil1986weak,KloPla-95,talay1990expansion,bally1996law}, see also \cite{milstein1994numerical,MilsteinT2004}

Although the diffusion models are well established, the
presence of jumps typically requires an additional justification.

In various application areas, jumps appear quite naturally: finance (jumps of stock prices), population biology
(jump processes appear as limits of Markov chains). Some of these models are well described by It\^o SDEs of the type $\di X=F(X)\,\di L$.
The weak convergence of the Euler scheme for SDEs with a jump component of finite intensity was studied by
\cite{mikulevicius1988time,kubilius2002rate}.
\cite{ProTal-97} established the first order  convergence of the Euler scheme in particular in case of $C^4_b$ coefficients, $C^4_b$-function $f$, and tails of the L\'evy
measure having finite 8th moments (they also have results for increasing $f$; then more moments are needed).
Further analysis was performed by \cite{JacodKMP-05} (Remark 2.7, 12th moments needed).
\cite{liu2000weak} studied the SDE driven by a BM and a PRM (all moments of $X$ are needed).
Recently weak approximations for SDEs with H\"older-continuous coefficients were studied by
\cite{mikulevicius2011rate,mikulevicius2012rate,mikulevicius2015weak}.

There is however another, (mechanical) point of view on SDE, which originates in the suppositon that
both the Brownian motion and jump processes are convenient mathematical idealizations of smooth real-world processes
(i.e.\ mechanical motions). This paradigm goes back to the Langevin who obtained a random motion of a heavy particle in a liquide
as an integral of a correlated Gaussian velocity process.

It turnes out that the idealized diffusion dynamics in such an approach is correctly described by the
Stratonovich SDEs which can be seen as limit of random non-autonomous ODEs in which the
Brownian motion by replaced by its (piece-wise) smooth approximations (Wong--Zakai approximations).

In the presence of jumps,
the Marcus (canonical) SDEs are extensions of Stratonovich SDEs for diffusions. As Stratonovich equations, they have
lot of useful (natural) properties such as
the Newton--Leibniz chain rule. They are also limits of continuous random ODEs obtained by pathwise
approximations of the driving L\'evy process by smooth functions
(the Wong--Zakai technique). For applications in Physics see e.g.\
\cite{ChePav-14,PavLiXuChe-15}.

Roughly speaking, jumps in the Marcus setting should be understood as idealizations of very fast motions along certain trajectories
determined by the physical parameters of the system.

Despite of these usefulness, numerical methods for Marcus SDEs are not well-developed.
Some partial results on the physical level of rigour can be found in
\cite{LiMinWang-13,LiMinWang-14err}.

The goal of this paper is to fill this gap
and construct an Euler-Maruyama type numerical scheme $\bar X$ on a discrete time grid of the size $h>0$, and to
establish the first order weak approximations $|\E f(X_T)-\E f(\bar X_T)|\leq Ch$ for a certain class of test functions $f$.
The main difficulty will consist in the treatment of the Marcus jump term,  which involves the analysis of a certain family of
non-linear ODEs and makes the problem different to the It\^o case.

\section{Setting and the main result}

On a filtered probability space $(\Omega,\rF,\mathbb F,\P)$ satisfying the usual hypotheses consider an $m$-dimensional
Brownian motion $W$ and an independent $m$-dimensional pure jump L\'evy process $Z$
with a characteristic triplet $(0, 0, \nu)$,
\ba
Z(t)=\int_0^t \int_{\|z\|\leq 1} z\, \tilde N(\di s,\di z)+ \int_0^t \int_{\|z\|> 1} z\, N(\di s,\di z).
\ea
For $d\geq 1$ consider vector-valued function
\ba
a(x)=\begin{pmatrix}
      a^1(x),\\\vdots\\a^d(x)
     \end{pmatrix},
\ea
and matrix-valued functions
\ba
% W_t&=\begin{pmatrix}
%      W_t^1,\\\vdots\\W_t^m
%     \end{pmatrix},
% \quad
% Z_t=\begin{pmatrix}
%      Z_t^1,\\\vdots\\Z_t^m
%     \end{pmatrix},\quad
% a(x)=\begin{pmatrix}
%      a^1(x),\\\vdots\\a^d(x)
%     \end{pmatrix},\\
% b(x)&=\begin{pmatrix}
%      b_{11}(x)&\cdots &b_{1m}(x)\\
%      \vdots&\ddots &\vdots\\
%      b_{d1}(x)&\cdots&b_{dm}(x)
%     \end{pmatrix},\qquad
% c(x)=\begin{pmatrix}
%      c_{11}(x)&\cdots &c_{1m}(x)\\
%      \vdots&\ddots &\vdots\\
%      c_{d1}(x)&\cdots&c_{dm}(x)
%     \end{pmatrix},\\
b(x)&=\begin{pmatrix}
     b^{1}_1(x)&\cdots &b^{1}_m(x)\\
     \vdots&\ddots &\vdots\\
     b^{d}_1(x)&\cdots&b^d_m(x)
    \end{pmatrix},\qquad
c(x)=\begin{pmatrix}
      c^{1}_1(x)&\cdots &c^{1}_m(x)\\
     \vdots&\ddots &\vdots\\
     c^{d}_1(x)&\cdots&c^d_m(x)
    \end{pmatrix},\\
c^i(x)&=\begin{pmatrix}c^i_1(x),\dots,c^i_m(x)\end{pmatrix} \quad \text{is the $i$-th row of the matrix $c(x)$, }i=1,\dots,d .
%
%
%
% b^i(x)&=\begin{pmatrix}b^i_1(x)&\cdots & b^i_m(x)\end{pmatrix} \quad \text{the $i$-th row, }i=1,\dots,d ,\\
% b_j(x)&=\begin{pmatrix}b^1_j(x)\\\cdots \\b^d_j(x)\end{pmatrix} \quad \text{the $j$-th column, }j=1,\dots,m, \\
% Db^i(x)&=\begin{pmatrix}
%     \displaystyle \frac{\partial}{\partial x^1} b^i_1(x )&\cdots &\displaystyle\frac{\partial }{\partial x^d}b^i_1(x)\\
%      \vdots&\ddots &\vdots\\
%    \displaystyle  \frac{\partial }{\partial x^1}b^i_m(x)&\cdots&\displaystyle\frac{\partial}{\partial x^d} b^i_m(x)
%     \end{pmatrix}\quad \text{the gradient matrix of the $i$-th row, }i=1,\dots,d ,\\
% % DG^l \cdot G &=\begin{pmatrix}
% %      \sum_{j=1}^d \frac{\partial G^l_1 }{\partial x^j} G^j_1&\cdots & \sum_{j=1}^d \frac{\partial G^l_1 }{\partial x^j} G^j_m\\
% %      \vdots&\ddots &\vdots\\
% %      \sum_{j=1}^d \frac{\partial G^l_m }{\partial x^j} G^j_1   &\cdots&  \sum_{j=1}^d \frac{\partial G^l_m }{\partial x^j} G^j_m
% %     \end{pmatrix},\\
% % \tr DG^l \cdot G &=  \sum_{i=1}^m  \sum_{j=1}^d \frac{\partial G^l_i }{\partial x^j} G^j_i
%     \quad
%  c(x)&=\begin{pmatrix}
%      c^1_1(x)&\cdots &c^1_m(x)\\
%      \vdots&\ddots &\vdots\\
%      c^d_1(x)&\cdots&c^d_m(x)
%     \end{pmatrix},\\
\ea

We consider a Marcus (canonical) SDE
\be
\label{e:SDEM}
X_t=X_0+\int_0^t a(X_s)\, \di s+\int_0^t b(X_s)\circ\di W_s+\int_0^t c(X_s)\diamond \di Z_s,\quad t\geq 0.
\ee
It can be  rewritten as an It\^o SDE driven by a Brownian motion and a Poissonian random measure:
% \ba
% \label{e:SDEI}
% X_t&=X_0+ \int_0^t \Big(a(X_s)+\frac12 D b(X_s)b(X_s)\Big)\, \di s
% +\int_0^t\int_{\|z\|\leq 1}\Big( \ophi^z(X_{s-})-X_{s-}-c(X_{s-})z  \Big)\,\nu(\di z)\,\di s\\
% &+\int_0^t b(X_s)\, \di W_s
% +\int_0^t \int_{\|z\|\leq 1}
% \Big(\ophi^z(X_{s-})-X_{s-}\Big)\, \tilde N(\di z,\di s)+ \int_0^t\int_{\|z\|> 1}  \Big(\ophi^z(X_{s-})-X_{s-}\Big)\,   N(\di z, \di s),
% \ea
in the coordinate form as
\ba
\label{e:SDEI}
X^i_t=x^i&+
\int_0^t a^i(X_{s-})\,\di s \\
&+   \sum_{j=1}^m \int_0^t b^i_j(X_{s})\,\di W^j_s
+  \frac12 \sum_{j=1}^m\sum_{l=1}^d\int_0^t \frac{\partial}{\partial x^l}b^i_j (X_{s})b^l_j(X_{s})  \,\di s\\
&+ \int_0^t\int_{\|z\|\leq 1} \Big(\ophi^z(X_{s-}) -X_{s-}  \Big)^{\!i} \, \tilde N(\di z,\di r) \\
&+ \int_0^t\int_{\|z\|\leq 1} \Big(\ophi^z(X_{s}) -X_{s} - c(X_{s})z \Big)^{\!i} \, \nu(\di z)\,\di r \\
&+ \int_0^t\int_{\|z\|>1} \Big(\ophi^z(X_{s-}) -X_{s-}  \Big)^{\!i} \, N(\di z,\di r), \\
\ea
where $\ophi^z(x)$ is a Marcus flow generated by the non-linear ordinary differential equation
\ba
&\begin{cases}
\label{e:ophi-eq}
&\displaystyle\frac{\di}{\di u} \ophi^z(u;x)=c(\ophi^z(u;x))z\\
&\displaystyle\ophi^z(0;x)=x,\quad u\in[0,1],\\
 \end{cases}\\
&\ophi^z(x):=\ophi^z(1;x).
\ea

For a complete account on Marcus SDEs see \cite{Marcus-78,Marcus-81,KurtzPP-95,Kunita-04,Applebaum-09}.
Note that the Marcus integral $\int_0^t c(X_s)\,\diamond\, \di Z_s$ cannot be represented a limit of Riemannian sums
(opposite to the Stratonovich integral), so that the SDE
\eqref{e:SDEM} should be understood via its It\^o representation \eqref{e:SDEI}.
% , which can be equivalently written as
% \ba
% \label{e:SDEI-mod}
% X_t&=X_0+ \int_0^t \Big(a(X_s)+\frac12 Db(X_s)b(X_s)\Big)\, \di s +\int_0^t b(X_s)\, \di W_s \\
% &+ \int_0^t c(X_{s-})\, \di Z_s
% +\sum_{s\leq t} \Big(\ophi^{\Delta Z_s}(X_{s-})-X_{s-} +c(X_{s-})\Delta Z_s   \Big).
% \ea

For $f\colon \bR^d\mapsto\bR$ we will use the uniform norm
\ba
\|f\|=\sup_{x\in\bR^d}|f(x)|.
\ea
For $x\in\bR^d$ (and $\bR^m$), we will work with the Euclidian norm
\ba
\|x\|=\Big((x^1)^2+\dots+(x^d)^2\Big)^{1/2}.
\ea

For a function $f\colon \bR^d\to \bR$ denote $\partial^\alpha f$ its partial derivative corresponding to a multiindex $\alpha$.
Let $Dc(x)$ be the gradient tensor of the mapping $x\mapsto c(x)$.
For each $x\in\bR^d$, we consider it as a linear operator $Dc(x)\colon \bR^m\to \bR^{d\times d}$ given by
\ba
Dc(x)z=\begin{pmatrix}
         \displaystyle \frac{\partial}{\partial x^1} \langle c^1(x),z\rangle &\cdots &\displaystyle\frac{\partial}{\partial x^d} \langle c^1(x),z\rangle\\
     \vdots&\ddots &\vdots\\
   \displaystyle  \frac{\partial}{\partial x^1} \langle c^d(x),z\rangle&\cdots&\displaystyle\frac{\partial}{\partial x^d} \langle c^d(x),z\rangle
       \end{pmatrix}.
\ea
Then we define
\ba
\|Dc(x)\|= \sup_{\|z\|\leq 1} \|Dc(x)z\|,
% =s_1 (Dc(x)z)
\ea
and let
\ba
\|Dc\|= \sup_{x\in\bR^d} \|Dc(x)z\|
\ea
% and
% the operator norm of $Dc(x)z$
% \ba
% \|Dc(x)z\|= \sup_{\|\xi\|\leq 1} \Big\{   \sum_{i,j=1}^d  \frac{\partial}{\partial{x^j}} \langle c^i(x),z\rangle\xi^i\xi^j   \Big\}
% =s_1 (Dc(x)z)
% \ea
% where $s_1 (Dc(x)z)$ is the greatest singular value of the matrix $Dc(x)z$.
For practical needs it is sometimes convenient to use the
the maximum entry norm of the gradient tensor
\ba
\|Dc\|_\ex= \max_{\substack{1\leq i,k\leq d \\ 1\leq j\leq m }}   \Big\|\frac{\partial}{\partial{x_k}}c^i_j(x) \Big\|.
\ea
Then we have
% \ba
% (Dc(x)z)_{ij}^2=\Big|\frac{\partial}{\partial x^j} \langle c^i(x),z\rangle\Big|^2
% =\Big|\sum_{k=1}^m \frac{\partial}{\partial x^j} c^i_k(x)z_k\Big|^2
% \leq \|Dc\|_\ex^2 \cdot \Big|\sum_{k=1}^m |z_k| \Big|^2
% \leq \|Dc\|_\ex^2 \cdot m \cdot \|z\|^2.
% \ea
% Hence
% \ba
% \|Dc(x)z\|\leq \Big(\sum_{i,j=1}^d (Dc(x)z)_{ij}^2\Big)^{1/2}  \leq  d\sqrt m\cdot  \|Dc\|_\ex \cdot \|z\|
% \ea
\ba
\label{e:De}
\|Dc(x) z\| \leq \|Dc\|\cdot \|z\|\leq d\sqrt m \cdot \|Dc\|_\ex \cdot \|z\| .
\ea

In this paper we make the following assumptions on the coefficients $a$, $b$ and $c$.

\noindent
% \textbf{LG:} linear growth of the coefficients:
% \ba
% \label{e:LG}
%  &|a(x)|+|b(x)|+|b(x)b'(x)|+ |c(x)|\leq K(1+|x|),\quad x\in\bR,\\
% \ea
\textbf{H$_{a,b,c}$:}
\ba
\label{e:BDab}
&a\in C^4(\bR^d,\bR^d)\quad \text{and}\quad  &&\|\partial^\alpha a^i\|<\infty , \quad 1\leq i\leq d, \quad 1\leq |\alpha|\leq 4;\\
&b\in C^4(\bR^d,\bR^{d\times m})  \quad \text{and}\quad   &&\|\partial^\alpha b^i_j\|<\infty ,\quad  1\leq i\leq d, \ 1\leq j\leq m,
\quad 1\leq |\alpha|\leq 4,\\
&&&\| b^i_j\cdot \partial^\alpha b^k_l\|<\infty, \quad 1\leq i,k\leq d, \ 1\leq j,l\leq m,
\quad 2\leq |\alpha|\leq 4 ;\\
&c\in C^4(\bR^d,\bR^{d\times m})  \quad \text{and}\quad   &&\|\partial^\alpha c^i_j\|<\infty ,\quad  1\leq i\leq d, \ 1\leq j\leq m,
\quad 1\leq |\alpha|\leq 4,\\
&&&\| c^i_j\cdot \partial^\alpha c^k_l\|<\infty, \quad 1\leq i,k\leq d, \ 1\leq j,l\leq m,
\quad 2\leq |\alpha|\leq 4 .
\ea
Under these conditions there is a unique global solution $\ophi^z$ of \eqref{e:ophi-eq} whose properties are studied in Appendix \ref{a:ophi}.

We consider a numerical scheme for the equation \eqref{e:SDEM} based on the Wong--Zakai approximations of the driving processes $W$ and $Z$.
For the time step $h>0$, let us approximate $W$ and $Z$ by polygonal curves with knots at $\{kh,W_{kh}\}_{k\geq 0}$, $\{kh, Z_{kh}\}_{k\geq 0}$,
namely we define the continuous time processes
\ba
% \label{SDE_M_h}
% \bar X_t&=x+\int_0^t \Big(a(\bar X_s)+b(\bar X_s) \dot W_s^h+c(\bar X_s)\dot Z_s^h\Big)\, \di s,\quad t\in[0,T],\\
% % &\eta_h(t)=h\left\lfloor h^{-1}t\right\rfloor, \quad \zeta_n(t)=\eta_h(t)+h,\\
W^h_t&=W_{kh}+\frac{t-kh}{h}\Big(W_{(k+1)h}-W_{kh}\Big),\quad t\in[kh,(k+1)h),\quad k\geq 0,\\
% &W^h_s=W_{[s]_h}+h^{-1}(s-[s]_h)\Big(W_{[s]_h+h}-W_{[s]_h}\Big),\\
% &Z^h_t=Z_{\eta_h(t)}+h^{-1}(t-\eta_h(t))\Big(Z_{\zeta_h(t)}-Z_{\eta_h(t)}\Big)
Z^h_t&=Z_{kh}+\frac{t-kh}{h}\Big(Z_{(k+1)h}-Z_{kh}\Big),\quad t\in[kh,(k+1)h),\quad k\geq 0,
\ea
and consider the sequence of random ODEs
\ba
\label{SDE_M_h}
\bar X_t&=x+\int_0^t \Big(a(\bar X_s)+b(\bar X_s) \dot W_s^h+c(\bar X_s)\dot Z_s^h\Big)\, \di s,\quad t\geq 0.
\ea
It is well known, see \cite{Marcus-78,Kunita-95}, that the approximations $\bar X$ converge to $X$, $h\to 0$ in the sense of convergence of finite
dimensional distributions.

Taking into account that $\dot W^h$ and $\dot Z^h$ are piece-wise constant, we obtain the discrete time scheme
$\bar X=(\bar X_{kh})_{k\geq 0}$ as follows.

For $\tau\geq 0$, $w,z\in\bR^m$, consider the ordinary differential equation
\ba
\label{e:eqpsi}
&\frac{\di}{\di u} \psi(u)=a(\psi(u))\tau+   b(\psi(u))w+c(\psi(u))z\\
&\psi(0)=x,\quad u\in[0,1],
\ea
which has a unique global solution under assumptions \textbf{H$_{a,b,c}$}.
Let
\ba
\label{e:psi}
\psi(x)=\psi(x; \tau,w,z):=\psi(1;x,\tau ,w,z).
\ea
The properties of $\psi$ are studied in Appendix \ref{a:psi}.

For the time step $h>0$, consider the Euler scheme
\ba
\label{e:Euler}
&\bar X_0=x,\\
&\bar X_{(k+1)h}=\psi(\bar X_{kh}; h, W_{(k+1)h}-W_{kh},Z_{(k+1)h}-Z_{kh}),\quad k\geq 0.
\ea
The goal of this paper is to establish the weak convergence rate of this numerical scheme. It is assumed that the increments of the Brownian motion and
the pure jump process $Z$ can be simulated exactly. We also do not take into account numerical errors which may arise
in the solution of ODE \eqref{e:eqpsi}.

%
% \subsection{Conditions}
%
%
% Let us assume that there are continuous functions $A=A(\tau)$, $B=B(w)$, $C=C(z)$ and $\alpha=\alpha(\tau)$, $\beta=\beta(w)$, $\gamma=\gamma(z)$
% such that
% \begin{align}
% \label{e:psiA}
% &|\psi(u;x;\tau,w,z)|\leq (1+|x|)\Big(A(\tau)+B(w)+C(z)\Big),\quad u\in[0,1],\\
% \label{e:psia}
% &\int_s^t \Big( a'(\psi(u;x;\tau,w,z))\tau+ b'(\psi(u;x;\tau,w,z))w  + c'(\psi(u;x;\tau,w,z))z \Big) \,\di u
% \leq \alpha(\tau)+\beta(w)+\gamma(z),\\
% &\qquad \quad 0\leq s\leq t\leq 1.
% \end{align}
% We also denote $A(z)=A(0,0,z)$ and $\alpha(z)=\alpha(0,0,z)$, and note that hence
% \ba
% \label{e:Aa}
% &|\ophi^z(u;x)| \leq (1+|x|)A(z),\quad u\in[0,1],\\
% &\int_s^t c'(\ophi^z(u;x)) z \,\di u \leq \alpha(z),\quad 0\leq s\leq t\leq 1.
% \ea
%
% \begin{exa}
% If the coefficients $a,b,c$ satisfy \emph{\textbf{H}}$_{a,b,c}$ and are \emph{bounded}
% then we can take
% \ba
% \alpha(\tau,w,z)&=\|a'\|\cdot \tau+\|b'\|\cdot |w|+\|c'\|\cdot |z|,\\
% % A(\tau,w,z)&=1+\|a\|\tau+\|b\||w|+\|c\||z|,\\
% A(\tau,w,z)&=1+\|a\|\cdot \tau+\|b\|\cdot |w|+\|c\|\cdot |z|.
% \ea
% \end{exa}
% \begin{exa}
% If the coefficients are unbounded, i.e.\ of \emph{linear growth} then
% then we can take
% \ba
% \alpha(\tau,w,z)&=\|a'\|\tau+\|b'\||w|+\|c'\||z|,\\
% A(\tau,w,z)&=1+C(\tau+|w|+|z|)\ex^{\alpha(\tau,w,z)} \quad \text{for some }C>0.
% \ea
% \end{exa}
%

Now we formulate the Assumptions and main results of this paper.

% \noindent
% \textbf{H$_{\ophi,A}$:}
% \ba
% &\int_{|z|>1} A(z)^4 \,\nu(\di z)<\infty
% \ea
%

\noindent
\textbf{H$_{\nu}$:} Assume that on the tails of the L\'evy measure $\nu$ satisfy
\ba
&\int_{\|z\|>1} \|z\|^3\cdot \ex^{8\|D c\|\cdot \|z \|}\,\nu(\di z)<\infty.
\ea
In view of \eqref{e:De}, Assumption \textbf{H$_{\nu}$} is granted by  the following condition which is easier to verify in practice:
\noindent
\textbf{H$'_{\nu}$:}
\ba
\int_{\|z\|>1} \|z\|^3\cdot \ex^{8d\sqrt m\cdot \|D c\|_\ex \cdot \|z \|}\,\nu(\di z)<\infty.
\ea

% \noindent
% \textbf{H$_{\nabla \psi,\nu}$:}
% \ba
% &\int_{|z|>1}(1+A(z))\ex^{\alpha(z)}\,\nu(\di z)<\infty \quad\text{(this one follows form the previous ones)},\\
% &\limsup_{h\to 0} \int_0^1\int_{|z|>1}\E (|\theta W_h|+|z+\theta \tilde Z_h|)^2 A(\theta h,\theta W_h,z +\theta \tilde Z_h)^4
% \ex^{6\alpha(\theta h,\theta W_h,z +\theta \tilde Z_h)}\,\nu(\di z)\,\di \theta<\infty
% \ea
%
%
%
% In both cases, Assumptions
% \textbf{H$_{\ophi,\nu}$}, \textbf{H$_{\nabla\ophi,\nu}$} and \textbf{H$_{\nabla \psi,\nu}$} are satisfied if
% \ba
% &\int_{|z|>1} z^6\ex^{10\|c'\| |z|}\,\nu(\di z)<\infty.
% \ea
% (the assumption on the second derivatives of $\psi$ holds due to the dominated convergence theorem).
%
%
% \newpage
%
%
%
%
%
%
%
% % In both cases, Assumptions
% % \textbf{H$_{\ophi,\nu}$}, \textbf{H$_{\nabla\ophi,\nu}$} and \textbf{H$_{\nabla \psi,\nu}$} are satisfied if
% \textbf{H$_{\nu}$:}
% \marginpar{check conditions on smoothness of the semigroup. We will need less moments}
% \ba
% &\int_{|z|>1} z^6\ex^{10\|c'\| |z|}\,\nu(\di z)<\infty.
% \ea
%
%

\begin{thm}%[Situ, Chapter 3.1]
\label{t:existence}
Assume that conditions \emph{\textbf{H$_{a,b,c}$}}
and  \emph{\textbf{H$_{\nu}$}} hold true.
Then for any $T>0$ there is a constant $C_T$ such that for any $x\in\bR$ the following holds.

\noindent
1. There is a unique strong solution $X=(X_t)_{t\in[0,T]}$ such that
\ba
\E_x \sup_{t\in[0,T]}\|X_t\|^4\leq C_T(1+\|x\|^4).
\ea
2. For any $h>0$ the numeric scheme $\{\bar X_{kh}\}_{0\leq kh\leq T}$ satisfies
\ba\label{4thmoment}
\E_x \sup_{0\leq kh\leq T}\|\bar X_{kh}\|^4\leq C_T(1+\|x\|^4).
\ea
\end{thm}
\begin{proof}
See Section \ref{s:existence}.
\end{proof}

The following result is interesting on its own.  Assume

\noindent
\textbf{H$_{\nabla\ophi,\nu}$:}
\ba
% &\int_{|z|>1} z^4\|\ophi^z(\cdot;\cdot )\|^4 \,\nu(\di z)<\infty,\\
&\int_{|z|>1} \|\nabla_x\ophi^z\|^4 \,\nu(\di z)<\infty,\\  %\quad \int_{|z|>1} \|\nabla_x\ophi^z(1;\cdot)\|^8 \,\nu(\di z)<\infty,\\
&\int_{|z|>1}\|\nabla^2_x\ophi^z\|^2 \,\nu(\di z)<\infty,\\
&\int_{|z|>1}\|\nabla^3_x\ophi^z\|^{4/3} \,\nu(\di z)<\infty,\\
&\int_{|z|>1}\|\nabla^4_x\ophi^z\| \,\nu(\di z)<\infty.
\ea

\begin{thm}
\label{t:C4}
Under assumptions \emph{\textbf{H$_{a,b,c}$}} and \emph{\textbf{H$_{\nabla\ophi,\nu}$}}, for
any $f\in C^4_b$, any $T>0$, there is $C>0$ such that for each $t\in[0,T]$ and any multiindex $\alpha$
\ba
\|\partial^\alpha \E_x f(X_t)\| \leq C,\quad 1\leq |\alpha| \leq 4.
\ea
\end{thm}
\begin{proof}
See Section \ref{s:C4}.
\end{proof}

\begin{rem}
Under assumptions  \textbf{H$_{a,b,c}$},
it follows from Lemma \ref{l:Aa} that  \textbf{H$_{\nu}$} implies \textbf{H$_{\nabla\ophi,\nu}$}.
\end{rem}

% Sometimes for convenience we will adopt the notation
% \ba
% \mathring a(x)=a(x)+\frac12 b(x)b'(x)
% \ea
%
%
% Let $\mu=\mu(z)$ be such that for $u\in[0,1]$
% \ba
% z\int_0^u c'(\ophi^z(s;x))\,\di s \leq \mu(z),\quad z\in\bR.
% \ea
% We assume that
% \ba
% \int_{|z|>1} z^4\,\ex^{4Mz}\,\di z<\infty \quad \text{and}\quad  \int_{|z|>1} z^4\,\ex^{4\mu(z)}\,\di z<\infty
% \ea
%
%
% \noindent
% \textbf{H$'_{c,\nu}$:}
% \ba
% &\int_{|z|>1} |\ophi^z(x)|^2\,\nu(\di z)<C(1+|x|^2),\\
% \sup_x&\int_{|z|>1} |\ophi^z_x(x)|^2\,\nu(\di z)<C,\\
% \sup_x&\int_{|z|>1} |\ophi^z_{xx}(x)|\,\nu(\di z)<C,\\
% \ea
% We assume that
% \ba
% \int_{|z|>1} z^4\,\ex^{4Mz}\,\di z<\infty \quad \text{and}\quad  \int_{|z|>1} z^4\,\ex^{4\mu(z)}\,\di z<\infty
% \ea

% \subsection{Wong--Zakai approximations}
%
%
%
% \begin{thm}
% Under assumptions \emph{\textbf{H$_{a,b,c}$}} and \emph{\textbf{H$_{\nu}$}},
% for each $T>0$ there is $C>0$ such that for each $h>0$ the unique solution $\{\bar X_{kh}\}_{0\leq kh\leq T}$ satisfies
% \ba
% \E_x \sup_{0\leq kh\leq T}|\bar X_{kh}|^4\leq C(1+|x|^4).
% \ea
% \end{thm}

%
%
%
% \subsection{The main result}
%
%

The main result of this paper is the  first order weak convergence rate of the Euler scheme \eqref{e:Euler}.
\begin{thm}
\label{t:main}
Let the assumptions \emph{\textbf{H$_{a,b,c}$}} and \emph{\textbf{H$_{\nu}$}}
% ,
% % \emph{\textbf{H$_{\nabla \ophi,\nu}$}}
% and \emph{\textbf{H$_{\ophi_{zz}}$}}
hold true.
Then for any $f\in C^4_b(\bR,\bR)$ and any $T>0$ there is a constant $C=C(T,f)$ such that
for any
$n\in \mathbb{N}$ and $h>0$ such that $nh\leq T$
\ba
\label{e:main}
|\E_x f(X_{nh})-\E_x f(\bar X_{nh})|\leq C \cdot nh^2,\quad x\in\bR^d.
\ea
\end{thm}
The proof of this theorem will be given in the following Sections.

\medskip

Eventually we comment on conditions    \textbf{H$_{a,b,c}$} and \textbf{H$_{\nu}$}, and the applicability of the numerical scheme.

\begin{rem}
Assumptions \textbf{H$_{a,b,c}$} are less restrictive than the assumptions in
\cite{ProTal-97} and \cite{JacodKMP-05}  where the coefficients are $C^4_b$ or smoother.
\end{rem}

\begin{rem}
Assumption \textbf{H$_{\nu}$} (or \textbf{H$'_{\nu}$}) requires existence of exponential moments of the L\'evy measure $\nu$ and looks
more restrictive than the assumptions in
\cite{ProTal-97} and \cite{JacodKMP-05}  where existence of high absolute moments (up to 32-th and higher) is demanded.
This occurs due to the non-linear nature of the ODE \eqref{e:ophi-eq}. Recall that the jump size of an It\^o SDE
$\di X_t=c(X_{t-})\,\di Z_t$
is
$\Delta X_t=c(X_{t-})\Delta Z_t$
and hence is a linear function of $\Delta Z_t$.
On the contrary, the jump size of the Marcus SDE
$\di X_t=c(X_{t})\diamond\di Z_t$
equals to $\Delta X_t= \ophi^{\Delta Z_t}(X_{t-})-X_{t-}$ and
is determined by a non-linear ODE \eqref{e:ophi-eq}. The best generic estimate for the size of this jump is given by the Gronwall
inequality. Hence exponential moments in the Marcus case serve as a natural analog of the conventional moments in the It\^o scheme.
For instance, assumptions \textbf{H$_{\nu}$} and \textbf{H$'_{\nu}$} are always satisfied for a L\'evy process $Z$ with bounded jumps.

In particilar cases one can find less restrictive assumptions on the moments of the L\'evy measure.
For instance one can show that in dimensions $d=m=1$ for the
equation $\di X_t=a(X_t)\,\di t+ b(X_t)\circ \di W_t+M X_t\diamond\di Z_t$, with $a,b\in C_b^4$ and
$M>0$, convergence \eqref{e:main} holds for any spectrally negative L\'evy process $Z$ with
$\nu((0,+\infty))=0$, and in particular for a spectrally negative stable L\'evy process.
However we were not able to find similar tractable sufficient conditions for convergence in general, especially in the multivariate case.
\end{rem}
\begin{rem}
The scheme \eqref{e:Euler} employs realizations of the increments of the L\'evy jump process $Z$. The list
of infinitely divisible distributions which can be simulated explicitly is rather short and includes $\alpha$-stable laws,
Gamma and variance Gamma distributions, as well as inverse Gaussian etc. We refer the reader to \cite[Section 3]{ProTal-97}
and \cite[Section II.6]{ContT-04} for more information on this subject and the description of the corresponding numerical algorithms.
\end{rem}

\bigskip

For the reader's convenience, in the following Sections \ref{s:existence}--\ref{s:C4} as well as in
the Appendices \ref{a:ophi} and \ref{a:psi} we assume that $d=m=1$. In the proof we will not use any of the geometrical advantages of the
one-dimensional setting and make this assumption just in order to simplify the notation  significantly.
The technical difficulties lie not in the higher dimensions of the state space but in the
analysis of the interplay of the terms $\di t$, $\circ \,\di W$ and $\diamond\, \di Z$ with the corresponding
terms in the approximation scheme  \eqref{e:Euler}. From this point of view, we are in a setting of a scalar equation driven by a three-dimensional
L\'evy process $(t,W_t,Z_t)$.

\section{Proof of Theorem \ref{t:existence}\label{s:existence}}

\begin{proof}
1. We denote
\ba
a^\diamond (x)
&=  a(x)+\frac12 b'(x)b(x) + \int_{|z|\leq 1}\Big( \ophi^z(x)-x-c(x)z  \Big)\,\nu(\di z)
+ \int_{|z|>1} \Big(\ophi^z(x)- x\Big)\,  \nu(\di z)
\ea
and write \eqref{e:SDEI} in dimension 1  as
\ba
\label{e:X}
X_t
% &=X_0+\int_0^t a^\diamond ( X_s)\, \di s +\int_0^t b( X_s)\, \di W_s\\
% &+\int_0^t \int_{|z|\leq 1} \Big(\ophi^z( X_{s-})- X_{s-}\Big)\, \tilde N(\di s, \di z)
% +\int_0^t \int_{|z|>1} \Big(\ophi^z( X_{s-})- X_{s-}\Big)\,  N(\di s, \di z)\\
&=X_0+\int_0^t a^\diamond ( X_s)\, \di s    % + \int_0^t \int_{|z|>1} \Big(\ophi^z( X_{s-})- X_{s-}\Big)\,  \nu(\di z)\,\di s\\
+\int_0^t b( X_s)\, \di W_s +\int_0^t \int_{\bR} \Big(\ophi^z( X_{s-})- X_{s-}\Big)\, \tilde N(\di s, \di z)
\ea
Due to Lemmas \ref{l:phi} and \ref{l:Aa}, the drift $a^\diamond$ is a Lipschitz continuous function,
and since
\ba
|\ophi^z(x)-x|\leq C(1+|x|)|z|\bI(|z|\leq 1)+  |x|(1+\ex^{\|c'\|\cdot |z|})\bI(|z|>1)
\ea
and
\ba
|\ophi^z(x)-x-\ophi^z(y)+y|\leq C|x-y|\cdot |z|\cdot \bI(|z|\leq 1)+ |x-y|(1+\ex^{\|c'\|\cdot |z|})\bI(|z|> 1) ,
\ea
existence and uniqueness of the strong solution $X$ with a finite fourth moment follows, e.g.\ from \cite[Theorem 3.1]{Kunita-04}.

2. The discrete time scheme $\bar X=(\bar X_{kh})_{k\geq 0}$ can be transformed to a continuous time process $\bar X_t, t\geq 0$ by taking
\ba
\label{e:Euler_cont}
\bar X_{t}=\psi(\bar X_{kh}; h, W_{t}-W_{kh},Z_{t}-Z_{kh}),\quad t\in [kh, (k+1)h].
\ea
Then, using the It\^o formula on the time interval $[kh, (k+1)h]$ and taking into account condition \textbf{H$_{\nu}$} and the properties of the mapping $\psi$ and its derivatives (see Lemma \ref{l:psi}), it is easy to show that
$$
\E (\bar X_{(k+1)h})^4-\E (\bar X_{kh})^4\leq Ch\Big(1+\E (\bar X_{kh})^4\Big), \quad k\geq 0
$$
with some constant $C$ which does not depend on $k$. This gives
$$
1+\E (\bar X_{kh})^4\leq (1+Ch)^{k}(1+\|x\|^4), \quad k\geq 0,
$$
which proves \eqref{4thmoment}.
\end{proof}

\section{One-step estimates \label{s:onestep}}

\begin{thm}
\label{t:onestep}
For any $f\in C^4_b$ there is a constant $C>0$ such that for any $h>0$ and $x\in\bR$
% \ba
% \Big|\E_{X_{kh}} f(X_{(k+1)h})  -\E_{X_{kh}} f(\bar X_{(k+1)h})\Big| \leq C h^2(1+X_{kh}^4)
% \ea
% and
\ba
\Big|\E_x f(X_h)  -\E_x f(\bar X_h)\Big| \leq C h^2 (1+x^4)
\ea
\end{thm}
The proof of this Theorem will be given in Section \ref{s:proof-1-step} after necessary preparations made in the next Section.

\subsection{Bounded jumps estimates}

Consider the pure jump L\'evy process
\ba
\tilde Z_t=\int_0^t \int_{|z|\leq 1} z\,\tilde N(\di z,\di s),
\ea
which is a zero mean L\'evy process with $|\Delta \tilde Z_t|\leq 1$. We
denote by $\tilde X$  the solution of the SDE
\ba
\label{e:hatX}
\tilde X_t&=x+\int_0^t a(\tilde X_s)\, \di s+\int_0^t b(\tilde X_s)\circ\di W_s+\int_0^t c(\tilde X_s)\diamond \di \tilde Z_s\\
&=  \int_0^t \tilde a(\tilde X_s)\, \di s
% +\int_0^t \int_{|z|\leq 1}\Big( \ophi^z(\tilde X_{s-})-\tilde X_{s-}-c(\tilde X_{s-})z  \Big)\,\nu(\di z)\,\di s\\
+\int_0^t b(\tilde X_s)\, \di W_s
+\int_0^t \int_{|z|\leq 1} \Big(\ophi^z(\tilde X_{s-})-\tilde X_{s-}\Big)\, \tilde N(\di s, \di z)\\
\ea
where we denote the \emph{effective drift} by
\ba
\tilde a(x)&=  a(x)+\frac12 b'(x)b(x) + \int_{|z|\leq 1}\Big( \ophi^z(x)-x-c(x)z  \Big)\,\nu(\di z).
\ea
We also introduce for convenience the \text{Stratonovich diffusion correction term}
\ba
\mathring a(x) =  a(x)+\frac12 b'(x)b(x) .
\ea
Note that due to Lemma \ref{l:phi}, $|\tilde a(x)|, |\mathring a(x)|\leq C(1+|x|)$ and $\tilde a', \mathring a'\in C^3_b(\bR,\bR)$.

\begin{lem}%[Situ, Chapter 3.1]
\label{l:boundXtilde}
Assume that conditions \emph{\textbf{H$_{a,b,c}$}} hold true.
Then for any $T>0$, any  $x\in\bR$ there is a unique strong solution $\tilde X=(\tilde X_t)_{t\in[0,T]}$. Moreover for each $p\geq 1$ and
$T>0$ there is a constant $K_{T,p}>0$ such that
\ba
\E_x \sup_{t\in[0,T]}|\tilde X_t|^p\leq K_{T,p}(1+|x|^p),\quad x\in\bR.
\ea
\end{lem}
\begin{proof}
 The proof is the same as in Theorem \ref{t:existence} with no conditions on big jumps $|z|>1$.
\end{proof}

The process $\tilde X$ is a strong Markov process with the generator
\ba
\label{e:L1}
\tilde Lf(x)= \mathring a(x)f'(x)+\frac12 b^{2}(x)f''(x) +\int_{|z|\leq 1}\Big(f(\ophi^z(x))-f(x)-f'(x)c(x)z \Big)\nu(\di z),\quad
f\in C^2_c(\bR,\bR).
\ea

% \ba
% f(\tilde X_t)&=f(x)
% + \int_0^t \int_{|z|\leq 1} f(\ophi^z(X_{s-}) - f( X_{s-} )- f'( X_{s-})\Big(\ophi^z(\tilde X_{s-})-\tilde X_{s-}\Big)\,\nu(\di z)\\
% &+f'(X_{s-})\int_{|z|\leq 1}\Big( \ophi^z(X_{s-})-X_{s-}-c(X_{s-})z  \Big)\,\nu(\di z)\\
% &=\int_0^t \int_{|z|\leq 1} f(\ophi^z(X_{s-})) - f( X_{s-} )
% +f'(X_{s-})\int_{|z|\leq 1}\Big(  -c(X_{s-})z  \Big)\,\nu(\di z)
% \ea
%

\begin{lem}
\label{l:L}
There is a constant $C>0$ such that for each $f\in C^2(\bR,\bR)$ with bounded first and second derivatives
 \ba
|\tilde L f(x)|\leq C\Big(\|f'\| + \|f''\|\Big)(1+x^2),\quad x\in\bR. % \leq  C_f(1+x^2).
\ea
\end{lem}
\begin{proof}
Taking into account the linear growth condition for $\mathring a$ and $b$ we get for some $C>0$
\ba
\Big|\mathring a(x)f'(x)+\frac12 b^{2}(x)f''(x)\Big|\leq C\|f'\| (1+|x|) + C\|f''\|(1+x^2).
\ea
To estimate the integral term in \eqref{e:L1} we note that
\ba
\label{e:ch}
f(\ophi^z(x))-f(x) - f'(x)c(x)z &
=z^2 \int_0^1 \int_0^s  \Big(f''c^2+ f'cc'\Big)(\ophi^z(u;x))  \,\di u\,\di s\,
% = \frac{z^2}{2} \cdot \partial_{zz}f(\ophi^z(x))\Big|_{z=\zeta}\\
% &= \frac{z^2}{2} \cdot \Big( f''(\ophi^\zeta(x))\ophi_z^\zeta(x)^2 + f'(\ophi^\zeta(x))\ophi_{zz}^\zeta(x)\Big)
\ea
and Lemma \ref{l:phi} yields
\ba
\Big|\int_{|z|\leq 1}\Big(f(\ophi^z(x))-f(x)-f'(x)c(x)z \Big)\nu(\di z)\Big|
\leq C( \|f''\| + \|f'\|) (1+x^2).
\ea
% \ba
% \label{e:ch}
% f(\ophi^z(x))-f(x)=f(\ophi^z(1;x))-f(x)= \int_0^1 f'(\ophi^z(u;x)) c( \ophi^z(u;x))z \,\di u , \\
% \ea
% and thus applying \eqref{e:ch} again to the function $x\mapsto f'(x)c(x)$ we get
% \ba
% \label{e:LI}
% \int_{|z|\leq 1}&\Big( f(\ophi(x;z))-f(x)-f'(x)c(x)z \Big)\,\nu(\di z)\\
% &=\int_{|z|\leq 1} z \int_0^1 \Big[ f'(\ophi^z(u;x)) c( \ophi^z(u;x))    -f'(x)c(x)\Big]\,\di s\, \nu(\di z)\\
% &=\int_{|z|\leq 1} z^2 \int_0^1 \int_0^u  \Big((f' c)'c\Big)(\ophi^z(r;x))  \,\di r\,\di s\, \nu(\di z).
% \ea
% Combining the linear growth estimates
% \ba
% &|(f'(x) c(x))'c(x)|= |f'(x)c'(x)c(x)+ f''(x)c^2(x)|\leq C\|f'\| (1+|x|)+ C \|f''\|(1+x^2)
% \ea
% with the bound
% \ba
% \sup_{|z|\leq 1, r\in[0,1]}|\ophi^z(r,x)|\leq C(1+ |x|)\\
% \ea
% from Lemma \ref{l:ophi-z} we obtain the result.
\end{proof}

\begin{lem}
\label{l:LL}
Let $f\in C^4_b(\bR,\bR)$. Then there is a constant $C>0$ such that for all $x\in\bR$
 \ba
&|\tilde L\tilde Lf(x)|\leq C(1+x^4).
\ea
\end{lem}
\begin{proof}
% The statement will follow from the estimates
% \ba
% |\tilde L\tilde Lf(x)|\leq C(1+x^4)\qquad \text{and}\qquad
% \E_x \sup_{t\in[0,1]} |\tilde X_t|^4\leq C(1+x^4).
% \ea
Denote $G(x):=\tilde Lf(x)$.

Then
\ba
(\tilde L\tilde Lf)(x)= (\tilde LG)(x)&= \mathring a(x)G'(x)+\frac12 b^{2}(x)G''(x)
+\int_{|z|\leq 1}\Big(G(\ophi^z(x))-G(x)-G'(x)c(x)z\Big)\,\nu(\di z).
\ea
We will show that $|G'(x)|\leq C(1+x^2)$, $|G''(x)|\leq C(1+x^2)$ and
\ba
\Big|\int_{|z|\leq 1}\Big(G(\ophi^z(x))-G(x)-G'(x)c(x)z\Big)\,\nu(\di z)\Big|\leq C(1+x^4).
\ea

\paragraph{1. The first derivative $G'$.} We have
\ba
\label{e:G1}
G'(x)= &
\mathring a'(x) f'(x)+ \Big(\mathring a(x)+ bb'(x)\Big)f''(x)+\frac12 b^2(x)f'''(x)\\
% \Big(a(x)+\frac{3}{2} bb'(x)\Big) f''(x)   +\frac{1}{2} b(x)^2 f^{(3)}(x)\\
% &a'(x)f'(x)+a(x)f''(x)+\frac{1}{2}[(bb'(x))'f'(x)+bb'(x)f''(x)+2bb'(x)f''(x)+b^{2}(x)f'''(x)]\\
&+\int_{|z|\leq 1}\Big(f'(\ophi^z(x))\ophi^z_x(x)-f''(x)c(x)z-f'(x)c'(x)z -f'(x)\Big)\,\nu(\di z)
\\&=(\tilde Lf')(x)+  \mathring a'(x)  f'(x)+ bb'(x) f''(x)
+\int_{|z|\leq 1}\Big(f'(\ophi^z(x))\big(\ophi^z_x(x)-1\big)-f'(x)c'(x)z\Big)\,\nu(\di z).
\ea
The term $\tilde Lf'$ is estimated by Lemma \ref{l:L} by $C(1+x^2)$, the term
$\mathring a'(x)  f'(x)$
 by $C$ and the term $bb'f''$ by $C(1+|x|)$.
To estimate the integral term,
we use Lemma \ref{l:phi} to get
\ba
\label{e:est1}
f'(\ophi^z(x))\big(\ophi^z_x(x)-1\big)-f'(x)c'(x)z&= f'(\ophi^z(x))\big(c'(x)z+\phi_x(1;x,z)\big)-f'(x)c'(x)z\\
&=c'(x)z^2 \int_0^1 (f'c)(\ophi^z(s;x))\,\di s + f'(\ophi^z(x)) \phi_x(1;x,z)
\ea
%
% we write the Taylor formula w.r.t.\ $z$ to get
% \ba
% \label{e:est1}
% f'(\ophi^z(x))&\big(\ophi^z_x(x)-1\big)-f'(x)c'(x)z\\
% &=z^2\Big[\frac12 \Big(f'''(\ophi^\xi(x)) (\ophi^\xi_z(x))^2+f''(\ophi^\xi(x)) \ophi^\xi_{zz}(x)\Big) ( \ophi^\xi_x(x)-1)\\
% &+f''(\ophi^\xi(x)) \ophi^\xi_z(x) \ophi^\xi_{xz}(x)
% +\frac{1}{2} f'(\ophi^\xi(x)) \ophi^\xi_{xzz}(x)\Big],
% % &=z^2\Big(f'(x)\partial_{zz}\partial_x \ophi^\zeta (x) +c'(x) f''(\ophi^\xi(x))\partial_z \ophi^\xi(x) \Big),
% \quad \xi=\xi(x,z),\ |\xi(x,z)|\leq |z|\leq 1.
% \ea
Taking into account the bounds from Lemma \ref{l:phi} we conclude that
the integral term is estimated by $C(1+|x|)$
%
%
%
% Taking into account the bounds from Lemma \ref{l:ophi-z} we conclude that there is a constant $C>0$ such that for all $|z|\leq 1$ and $x\in\bR$
% \ba
% \label{e:estim}
% &|\ophi_x^z(x)|\leq C,\quad |\ophi_{xz}^z(x) |\leq C,\quad |\ophi_{xzz}^z(x) |\leq C,\\
% &|\ophi_{z}^z(x)|\leq C(1+|x|),\\
% &|\ophi_{zz}^{z}(x)|\leq C(1+|x|^2).
% \ea
% Hence for $|z|\leq 1$ and $x\in \bR$
% \ba
% \Big|f'(\ophi^z(x))&\big(\ophi^z_x(x)-1\big)-f'(x)c'(x)z\Big|\leq C z^2(1+|x|^2).
% \ea
and eventually
\ba
\label{e:G11}
|G'(x)|\leq C  (1+|x|^2).
\ea

\paragraph{2. The second derivative $G''$.}
Straightforward differentiation yields
\ba
G''(x)&=\mathring a''(x)f'(x)+ \Big(2 \mathring a'(x)+ (b(x)b'(x))' \Big)f''(x)+ \Big(\mathring a(x)+ 2b(x)b'(x)\Big)f'''(x) + \frac12 b^2(x)f''''(x)\\
&+\int_{|z|\leq 1}\Big( f''(\ophi^z(x))(\ophi^z_x(x))^2    +f'(\ophi^z(x)) \ophi^z_{xx}(x)
-z f'(x) c''(x)- f''(x)(1+2 z c'(x)) -z f'''(x)c(x) \Big)\,\nu(\di z).
\ea
Recalling that
\ba
(\tilde Lf'')(x)&=\mathring a(x)f'''(x)+\frac12 b^{2}(x)f^{(4)}(x)+\int_{|z|\leq 1}\Big(f''(\ophi(x,z))-f''(x)-f'''(x)c(x)z \Big)\,\nu(\di z)
\ea
we can rewrite
\ba
\label{e:G2}
G''(x)&= (\tilde Lf'')(x)
+ \mathring a''(x)f'(x)+ \Big(2 \mathring a'(x)+ (b(x)b'(x))' \Big)f''(x)+ 2b(x)b'(x)f'''(x)  \\
&+\int_{|z|\leq 1}\Big( f''(\ophi^z(x))\Big(\ophi^z_x(x)^2 -1\Big)   +f'(\ophi^z(x))\ophi^z_{xx}(x)
-z f'(x) c''(x)- 2z f''(x) c'(x) \Big)\,\nu(\di z).
\ea
The first line of the previous formula is bounded by $C(1+x^2)$.
We estimate the integrand in its second line similarly to \eqref{e:est1} with the help of Lemma \ref{l:phi}. Denote for brevity
$\phi_x= \phi_x(1;x,z)$, $\phi_{xx}= \phi_{xx}(1;x,z)$.
\ba
f''(\ophi^z(x)) &\Big(\ophi^z_x(x)^2 -1\Big)  +f'(\ophi^z(x)) \ophi^z_{xx}(x) -z f'(x) c''(x)- 2z f''(x) c'(x)\\
&=f''(\ophi^z(x)) \Big( c'(x)^2z^2 +\phi_x^2+ 2c'(x)z+2c'(x)z\phi_x+2\phi_x        \Big)\\
&\qquad +f'(\ophi^z(x)) \Big( c''(x)z+\phi_{xx}  \Big) -z f'(x) c''(x)- 2z f''(x) c'(x)\\
&=2 z c'(x) \Big(f''(\ophi^z(x))-f''(x)\Big) +  z c''(x)\Big( f'(\ophi^z(x))-f'(x)\Big)\\
&\qquad +f''(\ophi^z(x)) \Big( c'(x)^2z^2 +\phi_x^2 +2c'(x)z\phi_x+2\phi_x        \Big)
+f'(\ophi^z(x)) \phi_{xx} \\
&=2 z^2 c'(x) \int_0^1 (f'''c)(\ophi^z(s;x))\,\di s +  z^2 c''(x)\int_0^1 (f'c)(\ophi^z(s;x))\,\di s\\
&\qquad +f''(\ophi^z(x)) \Big( c'(x)^2z^2 +\phi_x^2 +2c'(x)z\phi_x+2\phi_x        \Big)
+f'(\ophi^z(x)) \phi_{xx} ,
\ea
and hence the integral term in \eqref{e:G2} is bounded by $C(1+|x|)$.
%
% \ba
% f''(\ophi^z(x)) &\Big(\ophi^z_x(x)^2 -1\Big)  +f'(\ophi^z(x)) \ophi^z_{xx}(x)
% -z f'(x) c''(x)- 2z f''(x) c'(x) =z^2\cdot U(x,\xi),
% \ea
% where
% \ba
% U(x,\xi)&=\frac{1}{2}
% \Big[\Big(f^{(4)}(\ophi^\xi(x)) \ophi^\xi_z(x)^2+f^{(3)}(\ophi^\xi(x)) \ophi^\xi_{zz}(x)\Big) (\ophi^\xi_x(x)^2-1)
% +4 f^{(3)}(\ophi^\xi(x)) \ophi^\xi_z(x) \ophi^\xi_x(x) \ophi^\xi_{xz}(x)\\
% &+2 f''(\ophi^\xi(x)) \Big(\ophi^\xi_{xz}(x)^2+\ophi^\xi_x(x) \ophi^\xi_{xzz}(x)\Big)
% -f^{(3)}(\ophi^\xi(x)) \ophi^\xi_z(x)^2 \ophi^\xi_{xx}(x)-f''(\ophi^\xi(x)) \ophi^\xi_{zz}(x) \ophi^\xi_{xx}(x)\\
% &-2 f''(\ophi^\xi(x)) \ophi^\xi_z(x) \ophi^\xi_{xxz}(x)-f'(\ophi^\xi(x)) \ophi^\xi_{xxzz}(x)\Big],
% \quad \xi=\xi(x,z),\ |\xi(x,z)|\leq |z|\leq 1.
% \ea
% Lemma \ref{l:ophi-z} yields additionally to \eqref{e:estim} that there is a constant $C$ such that for $|z|\leq 1$ and $x\in\bR$
% \ba
% &|\ophi_{xxz}^z(x) \|\leq C,\quad |\ophi_{xxzz}^z(x) |\leq C,\\
% &|\ophi_{zz}^z (x)|\leq C(1+|x|^2).
% \ea
% Hence, $|U(x,\xi)|\leq C(1+|x|)$ as well as the integral term in \eqref{e:G2}, and e
Eventually
\ba
\label{e:G22}
|G''(x)|\leq C  (1+|x|^2).
\ea

\paragraph{3. The integral term of the generator.}

For $G(x):=\tilde Lf(x)$ we recall \eqref{e:ch}, \eqref{e:G11},  \eqref{e:G22}, and the estimate $\sup_{|z|\leq 1}|\ophi^z(x)|\leq C(1+|x|)$, to get
\ba
\Big|\int_{|z|\leq 1}\Big(G(\ophi^z(x))-G(x)-G'(x)c(x)z\Big)\,\nu(\di z)\Big|
\leq C(1+x^4).
\ea
\end{proof}

% and hence for each $x\in\bR$ and $f\in C^4_b(\bR,\bR)$
% \ba
% \E_x f(\tilde X_t)=f(x)+\int_0^t \E_x \tilde Lf(\tilde X_s)\,\di s.
% \ea

For the function $\psi=\psi(x;\tau,w,z)$ defined in \eqref{e:eqpsi} and \eqref{e:psi}, we introduce the process
\ba
Y_t=\psi(x;t,W_t,\tilde Z_t),\quad t\in[0,h].
\ea
% so that
% \ba
% &Y_0=x,\\
% &Y_h=\bar X_h.
% \ea
Since $\psi(\cdot;\cdot,\cdot,\cdot)\in C^4(\bR^4,\bR)$,
the It\^o formula implies that $Y$ is an It\^o process and
\ba
\E_x f(Y_t)=f(x)+\int_0^t \E Qf(\psi(x;s,W_s,\tilde Z_s))\,\di s,\quad f\in C^2_c(\bR,\bR),
\ea
with the generator
\ba
\label{e:Q}
Qg(\tau,w,z)=g_\tau(\tau,w,z)+\frac12 g_{ww}(\tau,w,z)
+\int_{|\xi|\leq 1} \Big(g(\tau,w,z+\xi)-g(\tau,w,z)-g_z(\tau,w,z)\cdot \xi\Big)\, \nu(\di \xi),
\ea
defined on smooth real-valued functions $g(\tau,w,z)$.

\begin{lem}
\label{l:L=Q}
Let $f\in C^2_b(\bR,\bR)$. Then
 \ba
 \label{e:L=Q}
\tilde Lf(x)=Qf(\psi(x;0,0,0)).
\ea
\end{lem}
\begin{proof}
For each $x\in\bR$, applying \eqref{e:Q} to $g(\tau,w,z):=f\circ \psi(x; \tau,w,z)$ we get
\ba
\label{e:Qf}
&Qf(\psi(x;\tau,w,z))=
f'(\psi(x;\tau,w,z)) \psi_\tau(x;\tau,w,z)\\
% &+{1\over 2}\prt_{w} (f'(\psi)\prt_w\psi(x;t,w,z)))\\
&+\frac12 f''(\psi(x;\tau,w,z)) \cdot (\psi_w(x;\tau,w,z))^2 +\frac12 f'(\psi(x;\tau,w,z))\psi_{ww}(x;\tau,w,z)))   \\
&+\int_{|\xi|\leq 1} \Big( f(\psi(x;\tau,w,z+\xi))-f(\psi(x;\tau,w,z))-f'(\psi(x;\tau,w,z))\psi_z(x;\tau,w,z)\cdot \xi\Big)\, \nu(\di \xi).
\ea
Recalling that $\psi(x;0,0,z)=\ophi^z(x)$ and $\psi(x;0,0,0)=x$, and
taking into account the formulae from Lemma \ref{l:psi} we find that
\ba
\psi_\tau(x;0,0,0)&=a(x),\\
\psi_w(x;0,0,0)&=b(x),\\
\psi_{ww}(x;0,0,0)&=bb'(x),\\
\psi_z(x;0,0,0)&=c(x),\\
\ea
and hence we get \eqref{e:L=Q}.
\end{proof}

\begin{lem}
\label{l:QQ}
Let $f\in C^4_b(\bR,\bR)$. Then there is a constant $C>0$ such that for any $\tau\geq 0$, $w\in\bR$, $z\in\bR$ and $x\in\bR$
\ba
&|QQf(\psi(x;\tau,w,z))  |\leq C(1+x^4)\cdot  \ex^{C(\tau+|w|+|z|)}.
\ea

\end{lem}
\begin{proof}
Denoting for brevity where it is possible $\psi=\psi(x;\tau,w,z)=\psi(\tau,w,z)$ or adopting
when necessary the notation $\psi(z):=\psi(x;\tau,w,z)$,
we apply the formula \eqref{e:Qf} for a $C^4_b$-function $f$ to get
\ba
Qf(\psi(\tau ,w,z))&=
f'(\psi) \psi_\tau
+\frac12f''(\psi) \cdot \psi_w^2 +\frac12 f'(\psi )\psi_{ww}    \\
&+\int_{|\xi|\leq 1} \Big( f(\psi(z+\xi))-f(\psi(z))-f'(\psi(z))\psi_z(z)\cdot \xi\Big)\, \nu(\di \xi)\\
&=f'(\psi) \psi_\tau
+\frac12f''(\psi) \cdot \psi_w^2 +\frac12 f'(\psi )\psi_{ww}
+\int_{|\xi|\leq 1} \xi^2 \int_0^1 \partial_{zz} f(\psi(z+\xi \theta))(1-\theta)\,\di \theta\, \nu(\di \xi).
\ea
With the help of \eqref{e:Q} we calculate
\ba
\label{e:QQ}
Q^2 f(\psi(\tau,w,z))
&=\prt_\tau Qf(\psi)+\frac12\prt_{ww}^2 Qf(\psi)
+\int_{|\xi|\leq 1} \Big(Qf(\psi(z+\xi))-Qf(\psi(z))-\partial_z Qf(\psi(z))\cdot \xi\Big)\, \nu(\di \xi)\\
&=\prt_\tau Qf(\psi)+\frac12\prt_{ww}^2 Qf(\psi)
+\int_{|\xi|\leq 1}     \xi^2 \int_0^1 \partial_{zz} Qf(\psi( z+\theta\xi))(1-\theta)\,\di \theta   \, \nu(\di \xi).
\ea
We estimate the summands in \eqref{e:QQ}.

\noindent
\textbf{1. $\partial_\tau Q f$.} First,
we write
\ba
\prt_\tau &Qf(\psi(\tau,w,z))=f''(\psi)\psi_\tau^2+ f'(\psi)\psi_{\tau\tau}\\
&+\frac12\Big( f'''(\psi)\cdot \psi_\tau \cdot (\psi_w)^2+ 2f''(\psi)\cdot \psi_w\cdot \psi_{\tau w}
+ f''(\psi)\cdot \psi_\tau\cdot \psi_{ww} +   f'(\psi)\cdot \psi_{\tau ww}  \Big)\\
&+\int_{|\xi|\leq 1} \xi^2 \int_0^1 \partial_{\tau zz} f(\psi(z+\xi \theta))(1-\theta)\,\di \theta\, \nu(\di \xi)
% &+\int_{|\xi|\leq 1} \Big(f'(\psi(z+\xi))\psi_\tau(z+\xi)   -
% f'(\psi(z))\psi_\tau(z)
% -f''(\psi)\psi_\tau(z)\psi_z(z)\cdot \xi
% -f'(\psi(z))\cdot \psi_{\tau z}(z) \cdot \xi
% \Big)\, \nu(\di \xi).
\ea
where for the inegral term we get
\ba
\partial_{\tau zz} f(\psi(\tau,w,z))&=
f'''(\psi) \psi_\tau \psi_z^2
+f''(\psi)  \psi_\tau \psi_{zz}
+2f''(\psi)  \psi_{\tau z} \psi_\tau
+f'(\psi) \psi_{\tau zz}.
\ea

% By Taylor's formula w.r.t.\ $\xi$, the integrand can be written as
% \ba
% & f'(\psi(z+\xi))\psi_\tau (z+\xi)   -
% f'(\psi(z))\psi_\tau(z) -f''(\psi) \psi_\tau(z)\psi_z(z)\cdot \xi
% -f'(\psi(z))\psi_{\tau z}(z) \cdot \xi \\
% &= \frac{\xi^2 }{2} \Big(
% f'''(\psi(\zeta)) (\psi_z(\zeta))^2 \psi_\tau(\zeta)
% +f''(\psi(\zeta)) \psi_{zz}(\zeta)\psi_\tau(\zeta)
% +2 f''(\psi(\zeta))\psi_z(\zeta) \psi_{\tau z}(\zeta)
% +f'(\psi(\zeta)) \psi_{\tau zz}(\zeta)
% \Big),\\
% &\quad \zeta=\zeta(x,z), \ |\zeta|\leq |z|,
% \ea
and hence in view of Lemma \ref{l:psi}
\ba
|\prt_\tau Qf(\psi(\tau,w,z))|\leq C  (1+|x|^3)\cdot (1+\tau+|w|+|z|)^2\cdot \ex^{C (\tau+|w|+|z|) }.
\ea

\noindent
\textbf{2. $\partial_{ww} Q f$.}
Analogously
\ba
\partial_{ww}Qf(\psi)&=  f'''(\psi) \psi_\tau \psi_w^2 + 2f''(\psi) \psi_{\tau w} \psi_w +  f''(\psi) \psi_{\tau} \psi_{ww}\\
&+ f'(\psi) \psi_{\tau ww} + \frac12 f^{(4)}(\psi)  \psi^4_w +3 f'''(\psi)  \psi_w^2 \psi_{ww}\\
&+\frac32 f''(\psi)  \psi_{ww}^2 +2 f''(\psi)  \psi_w\psi_{www} +\frac12 f'(\psi) \psi_{wwww}\\
&+\int_{|\xi|\leq 1} \xi^2 \int_0^1 \partial_{ww zz} f(\psi(z+\xi \theta))(1-\theta)\,\di \theta\, \nu(\di \xi),
\ea
where for the integral term we calculate
\ba
\partial_{wwzz} f(\psi)&= f^{(4)}(\psi) \psi_w^2\psi_z^2 \\
&+f'''(\psi) \psi_{ww}\psi_z^2
+4 f'''(\psi) \psi_{w}   \psi_{wz}   \psi_z
+ f'''(\psi) \psi_{w}^2   \psi_{zz} \\
&+ 2f''(\psi) \psi_{w}   \psi_{wzz}
+ 2f''(\psi) \psi_{wz}^2
+ (f'(\psi)+f''(\psi)) \psi_{ww}   \psi_{zz}
+ 2f''(\psi) \psi_{wwz}   \psi_{z},
\ea
which yields
\ba
|\partial_{ww}Qf(\psi(\tau,w,z))|\leq C(1+x^4) \cdot (1+\tau+|w|+|z|)^2\cdot  \ex^{C(\tau+|w|+|z|) }.
\ea

%
% \ba
% |\partial_{ww}Qf(\psi)|\leq  C(1+x^4) \cdot (1+\tau+|w|+|z|)^3\cdot  \ex^{C (\tau+|w|+|z|) },
% \ea
% and finally the straightforward differentialtion gives
% \ba
% \partial_{zz}\Big(f'(\psi)\psi_\tau & +{1\over 2}f''(\psi) \cdot (\prt_w\psi)^2 +\frac12 f'(\psi )\psi_{ww} \Big)=
% \frac{1}{2} \Big(f^{(4)}(\psi) \psi_\tau^2+f^{(3)}(\psi) \psi_{zz}\Big) \psi_w^2+2 f^{(3)}(\psi) \psi_z \psi_w \psi_{wz}\\
% &+\frac{1}{2} f''(\psi) \left(2 \psi_{wz}^2+2 \psi_w \psi_{wzz}\right)+\frac{1}{2} \left(f^{(3)}(\psi) \psi_{z}^2+f''(\psi) \psi_{zz}\right) \psi_{ww}\\
% &+f''(\psi) \psi_z \psi_{wwz}+\frac{1}{2} f'(\psi) \psi_{wwzz}+\left(f^{(3)}(\psi) \psi_{z}^2+f''(\psi) \psi_{zz}\right) \psi_t
% +2 f''(\psi) \psi_z \psi_{tz}+f'(\psi) \psi_{tzz}
% \ea
% and the Taylor formula yields again
% \ba
% \partial_{zz}\Big(& f(\psi(x;t,w,z+\xi))-f(\psi(x;t,w,z))-f'(\psi(x;t,w,z))\partial_z \psi(x;t,w,z)\cdot \xi\Big)\\
% &=\frac{\xi^2}{2} \Big(f^{(4)}(\psi) \psi_z^4+6 f^{(3)}(\psi) \psi_z^2 \psi_{zz}+3 f''(\psi) \psi_{zz}^2+4 f''(\psi) \psi_z \psi_{zzz}
% +f'(\psi) \psi_{zzzz}\Big)\Big|_{z=\zeta}
% \ea

\noindent
\textbf{3. $\partial_{zz} Q f$.}
We determine the derivatives
\ba
\partial_{zz}& \Big( f'(\psi) \psi_\tau +\frac12 f''(\psi) \cdot \psi_w^2 +\frac12 f'(\psi )\psi_{ww}     \Big)\\
&=f'(\psi)\psi_\tau\psi_{zz} + f''(\psi)\psi_\tau\psi_{z}^2+ (f'(\psi)+f''(\psi))\psi_{\tau z}\psi_z+f'(\psi)\psi_{\tau zz}\\
&+\frac12 f^{(4)}(\psi)\psi_w^2\psi_z^2 + 2f'''(\psi)\psi_w\psi_z\psi_{wz} + \frac12 f'''(\psi)\psi_w^2\psi_{zz}+f''(\psi)\psi_{wz}^2 +f''(\psi)\psi_w\psi_{wzz}\\
&+\frac12 f'''(\psi)\psi_{z}^2\psi_{ww}+f''(\psi)\psi_z\psi_{wwz}+ \frac12 \psi''(\psi)\psi_{ww}\psi_{zz}
+\frac12f'(\psi)\psi_{wwzz},
\ea
and
\ba
\partial_{zzzz} f(\psi)= f^{(4)}(\psi) \psi_z^4 + 6 f'''(\psi) \psi_z^2\psi_{zz}
+3 f''(\psi) \psi_{zz}^2 + 4  f''(\psi) \psi_{z}\psi_{zzz} + f'(\psi) \psi_{zzzz}
\ea
and apply Lemma \ref{l:psi} to get
\ba
|\partial_{zz} Qf(\psi(x;t,w,z))|\leq C(1+x^4) \cdot (1+\tau+|w|+|z|)^3\cdot  \ex^{C (\tau+|w|+|z|) }.
\ea
%
%
%
% Finally taking into account that for each $m\geq 1$
% \ba
% \sup_{t\in[0,1]}\E \ex^{m(\|a'\| t+\|b'\||W_t|+\|c'\||\tilde Z_t|)}<\infty
% \ea
% we obtain the estimate.
\end{proof}

\begin{lem}
\label{l:small}
For any $f\in C^4_b(\bR,\bR)$ there is a constant $C>0$ such that for any $h\geq 0$ and any $x\in\bR$
\ba
|\E_x f(\tilde X_{h})-\E_x f( \psi(x;h,W_h,\tilde Z_h)  )|\leq C(1+x^4)h^2.
\ea
\end{lem}
\begin{proof}
 Applying the It\^o formula twice we get
\ba
\E_x f(\tilde X_{h})-\E_xf( \psi(x;h,W_h,\tilde Z_h) )&= \int_0^h \E_x \tilde Lf(\tilde X_s)\,\di s- \int_0^h \E_x Qf(\psi(x;s,W_s,\tilde Z_s))\,\di s\\
&=h \tilde Lf(x) -h Q f(\psi(x;0,0,0))\\
&+   \int_0^h\int_0^s \E_x \tilde L\tilde Lf(\tilde X_r)\,\di r\, \di s- \int_0^h\int_0^s \E QQf(\psi(x;r,W_r,\tilde Z_r))\,\di r\,\di s,
\ea
and hence by Lemma \ref{l:L=Q} and H\"older's inequality for any $p>1$
\ba
\Big|\E_x f(\tilde X_{h})-\E_xf( \psi(x;h,W_h,\tilde Z_h) )\Big|
&\leq h^2 \sup_{r\in [0,h]} \E_x |\tilde L\tilde Lf(\tilde X_r)| + h^2 \sup_{r\in[0,h]}\E |QQf(\psi(x;r,W_r,\tilde Z_r))|\\
&\leq C h^2 \Big(1+ \sup_{r\in [0,h]} \E_x |\tilde X_r|^4\Big) + C h^2  \sup_{r\in [0,h]} \E_x (1+|\tilde X_r|^4)  \ex^{C(r+|W_r|+1)}\\
&\leq C h^2 (1+ |x|^4) + C h^2  \sup_{r\in [0,h]} \Big(\E_x (1+|\tilde X_r|^4)^p\Big)^{1/p}  \Big(\E \ex^{\frac{pC}{p-1}(r+|W_r|+1)}\Big)^{(p-1)/p}\\
&\leq  C h^2 (1+ |x|^4).
\ea
\end{proof}

\subsection{One-step estimate. Proof of Theorem \ref{t:onestep}\label{s:proof-1-step}}

% \begin{lem}
% \label{l:N0}
% For any $f\in C^4_b(\bR,\bR)$ there is $C>0$ sich that for any $h>0$, any $x\in\bR$
% \ba
% \E_x |f(X_h)-f(\bar X_h)|\leq C h^2(1+x^4).
% \ea
% \marginpar{maybe need more less 4}
% \end{lem}
\begin{proof}
Decompose the jump process $Z$ into a sum
\ba
Z_t=\tilde Z_t+ \sum_{k=0}^{N_t} J_k.
\ea
Assume from the very beginning that $\lambda=\nu(|z|>1)>0$.
Denote $\sigma:=\sigma_1$, the first jump time of $t\mapsto \int_0^t \int_{|z|>1}N(\di z,\di s)$, $J=J_1$ the size of the first large jump.
First note, that $\P(\tau\leq t|N_h=1)=t/h$, $t\in[0,h]$, and  $\P(J\in A|N_h=1)=\nu(A\cap\{|z|>1\})/\nu(|z|>1)$.

For each $x\in\bR$
% \ba
% &\E_x f(X_h)=\E_x \Big[f(X_h)\Big|N_h=0\Big]\P( N_h=0)+ \E_x \Big[f(X_h)\Big|N_h=1\Big]\P( N_h=1)+\E_x \Big[f(X_h)\Big|N_h\geq 2\Big]\P( N_h\geq 2),\\
% &\E_x f(\bar X_h)= \E_x \Big[f(\bar X_h)\Big|N_h=0\Big]\P( N_h=0)+ \E_x \Big[f(\bar X_h)\Big|N_h=1\Big]\P( N_h=1)
% +\E_x \Big[f(\bar X_h)\Big|N_h\geq 2\Big]\P( N_h\geq 2).
% \ea
\ba
|\E_x f(X_h)&-\E_x f(\bar X_h)|\\
&\leq |\E_x f(\tilde X_h)-\E_x f( \psi(x;h,W_h,\tilde Z_h) )|
+\E_x\Big[|f(X_h)-f(\bar X_h)|\Big|N_h=1\Big]\P(N_h=1)
+ 2\|f\|\P(N_h\geq 2).
\ea
The first summand is estimated by Lemma \ref{l:small} by $C(1+x^4)h^2$, the third has the order $h^2$.
Let us estimate the second summand.

First note that $\P(N_h=1)\leq Ch$.
Then, on the event $\{N_h=1\}$, the solution $X_h$ can be represented as a composition
\ba
X_h(x)=\tilde X_{\sigma,h}\circ \ophi^J \circ \tilde X_{0,\sigma-}(x)
\ea
and hence
\ba
f(\bar X_h(x))- f(X_h(x))
&= f(\ophi^J(x))-f( \tilde X_{\sigma,h}\circ \ophi^J(x)  )\\
&+  f( \tilde X_{\sigma,h}\circ \ophi^J(x) )  -f( \tilde X_{\sigma,h}\circ \ophi^J \circ \tilde X_{0,\sigma-}(x)  ) \\
&+f(\bar X_h(x))- f(\ophi^J(x)).
\ea

\noindent
\textbf{Step 1.}
Desintegrating the laws of $\sigma$, $J$ and $\tilde Z$ we obtain from the It\^o formula, Lemma \ref{l:L} and Assumption \textbf{H$_{\nu}$}
\ba
&|\E f( \tilde X_{\sigma,h}\circ \ophi^J(x) )  - \E f( \ophi^J(x)  )|\\
&\leq \frac{1}{\lambda h}\int_0^h \int_{|z|>1} |\E f(\tilde X_{h-s}( \ophi^z(x) )) -f(\ophi^z(x) )|\,\nu(\di z)\,\di s\\
&\leq \frac{1}{\lambda h }\int_0^h \int_{|z|>1}  \int_0^{h-s}    \E_{\ophi^z(x)} |\tilde L f(\tilde X_r)|\,\di r\,\nu(\di z)\,\di s\\
&\leq \frac{C}{\lambda h }\int_0^h \int_{|z|>1}  \int_0^{h-s}    \E_{\ophi^z(x)} (1+ |\tilde X_r|^2) \,\di r\,\nu(\di z)\,\di s\\
&\leq  \frac{C_1}{\lambda h}\int_0^h  \int_{|z|>1} h\Big(1+ |\ophi^z(x)|^2 \Big)\,\nu(\di z)\,\di s\\
&\leq C_2 h \cdot \Big( 1+\int_{|z|>1} |\ophi^z(x)|^2\,\nu(\di z)\Big)\\
&\leq C_3 h (1+x^2).
\ea

\noindent
\textbf{Step 2.} Acting similarly we estimate
\ba
|\E f( \tilde X_{\sigma,h}\circ \ophi^J(x) )  &- \E f( \tilde X_{\sigma,h}\circ \ophi^J \circ \tilde X_{\sigma-}(x)  )|\\
&\leq  \E\Big[ \Big|\E f( \tilde X_{\sigma,h}\circ \ophi^J(x) )
- \E f( \tilde X_{\sigma,h}\circ \ophi^J \circ \tilde X_{\sigma-}(x)  )\Big|\Big|  \rF_{\sigma}     \Big]\\
&\leq  \E_x\Big|   \E_{\ophi^J(x)} f( \tilde X_{h-\sigma})
- \E_{ \ophi^J(\tilde X_{\sigma-})  } f( \tilde X_{h-\sigma}  )  \Big|\\
% &=  \E\Big[ |\tilde P_{h-\sigma} f(\ophi^J(x) )  - \tilde P_{h-\sigma} f(\ophi^J ( \tilde X_{0,\sigma-}(x)  )|\Big|  \rF_{\sigma}     \Big]\\
&= \frac{1}{h}\int_0^h  \E_x\Big|   \E_{\ophi^J(x)} f( \tilde X_{h-s})
- \E_{ \ophi^J(\tilde X_{s-})  } f( \tilde X_{h-s}  )  \Big|\,\di s\\
&= \frac{1}{h}\int_0^h \int_{|z|>1} \E_x\Big|   \E_{\ophi^z(x)} f( \tilde X_{h-s})
- \E_{ \ophi^z(\tilde X_{s-})  } f( \tilde X_{h-s}  )  \Big|\,\nu(\di z)\,\di s\\
&= \frac{1}{h}\int_0^h \int_{|z|>1} \E_x\Big|   \tilde f^{h-s}( \ophi^z(x))
-  \tilde f^{h-s}( \ophi^z(\tilde X_{s-}) )   \Big|\,\nu(\di z)\,\di s\\
\ea

\ba
\E f( \tilde X_{\sigma,h}\circ \ophi^J(x) )  &- \E_x f( \tilde X_{\sigma,h}\circ \ophi^J \circ \tilde X_{\sigma-}  )\\
&\leq  \E\Big[\E f( \tilde X_{\sigma,h}\circ \ophi^J(x) )
- \E f( \tilde X_{\sigma,h}\circ \ophi^J \circ \tilde X_{\sigma-}(x)  )\Big|  \rF_{\sigma}     \Big]\\
&\leq  \E_x\Big[   \E_{\ophi^J(x)} f( \tilde X_{h-\sigma})
- \E_{ \ophi^J(\tilde X_{\sigma-})  } f( \tilde X_{h-\sigma}  )  \Big]\\
&= \frac{1}{h}\int_0^h \int_{|z|>1} \E_x\Big[   \E_{\ophi^z(x)} f( \tilde X_{h-s})
- \E_{ \ophi^z(\tilde X_{s-})  } f( \tilde X_{h-s}  )  \Big]\,\nu(\di z)\,\di s\\
\ea

Denote
\ba
\tilde f^{h-s}(x)=\E_x f( \tilde X_{h-s})  .
\ea
Since by Theorem \ref{t:C4}
\ba
\sup_{t\leq T}\Big( \|\tilde f^{t}_x\| +\|\tilde f^{t}_{xx}\| \Big)<C
\ea
we can calculate
\ba
\|\partial_x \tilde f^{t}(\ophi^z(x))\|&\leq C\cdot \|\ophi^z_x\|,\\
\|\partial_{xx} \tilde f^{t}(\ophi^z(x))\|&\leq C\Big(\|\ophi^z_x\|^2+\|\ophi_{xx}^z\|\Big).
\ea

Then for each $s\in[0,h]$ the It\^o formula and Lemma \ref{l:L} imply
\ba
\Big|\E_x \tilde f^{h-s}(\ophi^z(\tilde X_{s-}) ) - \tilde f^{h-s}(\ophi^z(x) ) \Big|
&\leq \int_0^s \E_x |\tilde L\tilde f^{h-s}(\ophi^z(\tilde X_{r}) )|\,\di r\\
&\leq C\cdot h\cdot  \Big(\|\ophi^z_x\|^2+\|\ophi_{xx}^z\|\Big) \cdot   (1+\sup_{r\in[0,h]}\E_x |\tilde X_r|^2)\\
&\leq C\cdot h\cdot  \Big(\|\ophi^z_x\|^2+\|\ophi_{xx}^z\|\Big) \cdot   (1+x^2).
\ea
Hence Assumption \textbf{H$_{\nu}$} yields
\ba
|\E f( \tilde X_{\sigma,h}\circ \ophi^J(x) )  - \E f( \tilde X_{\sigma,h}\circ \ophi^J \circ \tilde X_{\sigma-}(x)  )|
\leq Ch(1+x^2).
\ea

% \newpage

\noindent
\textbf{Step 3.}
Recall that $\bar X_h(x)=\psi(x;h,W_h,J+\tilde Z_h)$.
The Taylor expansion of $\psi=\psi(x;\tau,w,J+\xi )$ for a fixed $x$ at $(0,0,J)$ yields
\ba
\label{e:2nd}
f(\psi(x,h,w,J+\xi))&=f(\psi(x,0,0,J))\\
&+f'(\psi(x,0,0,J))\Big(\psi_\tau(x;0,0,J) h+ \psi_w(x;0,0,J) w + \psi_z(x;0,0,J)\xi  \Big)\\
&+R(x;h,w,J+\xi),
\ea
with the remainder term
\ba
R(x;h,w,J+\xi)
&=\frac12 \int_0^1 f''(\psi(\theta))\Big(\psi_{\tau\tau}(\theta)h^2+2\psi_{\tau w}(\theta)hw+2\psi_{\tau z}(\theta)h\xi+\psi_{ww}(\theta)w^2
+2\psi_{wz}(\theta)w\xi+\psi_{zz}(\theta)\xi^2\Big) \,\di \theta\\
&=R_1+\cdots+R_6.
\ea
where we write $\psi(\theta):=\psi(x;\theta h,\theta w,\theta\xi+J)$.

Due to the independence of $\tilde Z$, $J$ and $W$, $\E W_h=\E \tilde Z_h=0$
we get that the mean value of the second line in \eqref{e:2nd} vanishes.

To estimate the remainder term we have to estimate six terms with the help of \eqref{e:psi2}. Thus
\ba
\E |R_1|&\leq h^2 \|f''\| \int_0^1 \E |\psi_{\tau\tau}(x;\theta h,\theta W_h,\theta\tilde Z_h+J)| \,\di \theta\\
&\leq h^2 \|f''\| C(1+x^2) \E (2+h+| W_h| + |J |)\ex^{5(\|a'\|h+\|b'\||W_h|+\|c'\|(|J|+1))}\\
&\leq C h^2(1+x^2).
\ea

Analogously, the terms $R_2$ and $R_3$ are bounded by $C h(1+x^2)$. Further,
\ba
\E |R_4|\leq \|f''\|\int_0^1 &\E |\psi_{wz}(x;\theta h,\theta W_h,\theta\tilde Z_h+J)|\cdot W_h^2\,\di \theta\\
 &\leq  \|f''\| C(1+x^2) \E \Big[ W_h^2 (2+h+| W_h| + |J |)\ex^{5(\|a'\|h+\|b'\||W_h|+\|c'\|(|J|+1))}\Big]\\
&\leq C h(1+x^2),
\ea
where the factor $h$ essentially comes from the term $W_h^2$. The $R_2$ and $R_3$ are bounded by $C h(1+x^2)$ in a silmilar way.
\end{proof}

\section{Main estimates and the proof of Theorem \ref{t:main}}

According to Markov property of $X$, for each $t\in[0,T]$ and any bounded measurable $f$
\ba
\E_x f(X_T)= \E_x \E_{X_{T-t}} f(X_t)=\E_x f^t(X_{T-t}),
\ea
where
\ba
f^t (x) := \E_x f(X_t).
\ea
Let $0\leq nh\leq T$.
Denote
\ba
u_k(x):= \E_x f^{kh}(\tilde{X}_{T-kh}).
\ea
Then,
\ba
\E_x f(X_T)&=u_n,\\
\E_x f(\bar{X}_T)&=u_0,
\ea
and we have the following chaining representation
\ba
\label{chain}
\E_x f(X_{nh})-\E_x f(\bar{X}_{nh})&=\sum_{k=1}^n (u_k-u_{k-1})\\
&=\sum_{k=1}^n\Big(\E_x f^{kh}(\bar{X}_{nh-kh})-\E_x f^{(k-1)h}(\bar{X}_{nh-kh+h})\Big).
\ea
Observe that
\be
\label{fk-1h}
\E_x f^{(k-1)h}(\bar{X}_{(n-k+1)h}) = \E_x \E_{\bar{X}_{(n-k)h}} f^{(k-1)h}(\bar{X}_h),
\ee
and, using the property
\ba
f^{kh}(y) = \E_y f^{(k-1)h}(X_h),
\ea
we have that
\be
\label{fkh}
\E_x f^{kh}(\bar{X}_{(n-k)h}) = \E_x \E_{\bar{X}_{(n-k)h}} f^{(k-1)h}(X_h).
\ee
Combining \eqref{chain}), \eqref{fk-1h} and \eqref{fkh}, we finally have
\ba
\label{chain_fin}
\E_x f(X_{nh})-\E_x f(\bar{X}_{nh})=\sum_{k=1}^n\E_x \Big(\E_{\bar{X}_{(n-k)h}} f^{(k-1)h}(X_h)-\E_{\bar{X}_{(n-k)h}} f^{(k-1)h}(\bar{X}_h)\Big).
\ea
By Theorem \ref{t:onestep} and the $4$th moment bound \eqref{4thmoment} from Theorem \ref{t:existence},
\ba
\E_x \Big|\E_{\bar{X}_{(n-k)h}} f^{(k-1)h}(X_h)-\E_{\bar{X}_{(n-k)h}} f^{(k-1)h}(\bar{X}_h)\Big|
\leq
C_{1} h^2 (1+ \E_x|\bar{X}_{(n-k)h}|^4)\leq C_{2} h^2.
\ea

\section{$C^4$-smoothness of the Marcus semigroup\label{s:C4}. Proof of Theorem \ref{t:C4}}

We separate the proof in two parts. First, we prove the required statement in the case $\nu(|z|>1)=0$; that is, for $X=\tilde X$.
We consider all the derivatives of $f^t$ till the order $4$:
\begin{align}
\label{e:f1}
\partial_x f^t(x) &= \E_x\Big( f^{\prime}(X_t)\partial_x X_t \Big),\\
\label{e:f2}
\partial_{xx}f^t(x) &= \E_x\left( f^{\prime\prime}(X_t)(\partial_xX_t)^2 \right) + \E_x\left( f^{\prime}(X_t)\partial_{xx}X_t \right),\\
\label{e:f3}
\partial_{xxx}f^t(x) &= \E_x\left( f^{\prime\prime\prime}(X_t)(\partial_xX_t)^3 \right)
+ 3\E_x\left( f^{\prime\prime}(X_t)(\partial_{x}X_t)(\partial_{xx}X_t) \right) + \E_x\left( f^{\prime}(X_t)\partial_{xxx}X_t \right),\\
\label{f4}
\partial_{xxxx}f^t(x) &= \E_x\left( f^{(4)}(X_t)(\partial_xX_t)^4 \right) + 6\E_x\left( f^{\prime\prime\prime}(X_t)(\partial_{x}X_t)^2(\partial_{xx}X_t) \right)\\
&+ 3\E_x\left( f^{\prime\prime}(X_t)(\partial_{xx}X)^2 \right)+ 7\E_x\left( f^{\prime\prime}(X_t)(\partial_{x}X_t)(\partial_{xxx}X_t) \right)\nonumber\\
&+ \E_x\left( f^{\prime}(X_t)\partial_{xxxx}X_t \right).\nonumber
\end{align}

Then the required statement follows from
\begin{prp}
\label{pPT}
Let $\nu(|z|>1)=0$ and \emph{\textbf{H$_{a,b,c}$}} holds. Then for any $p>1$, $T<\infty$
\ba\label{p-bound}
\sup_{t\leq T, x\in \bR}\E |\partial_x^k X_t(x)|^{p}<\infty, \quad k=1, \dots, 4.
\ea
\end{prp}

Proposition \ref{pPT} has the same spirit with \cite[Lemma~4.2]{ProTal-97}. However, the above result is not applicable here
directly, because the It\^o form of the Marcus SDE
\ba\label{e:SDEtrun}
\di  X_t&=\tilde a(X_t)\, \di t
+b(X_t)\, \di W_t
+\int_{|z|\leq 1} \Big(\ophi^z(X_{t-})-X_{t-}\Big)\, \tilde N(\di t, \di z),
\ea
 contains the intergal w.r.t.\ the compensated Poisson random measure,
while \cite{ProTal-97} deal with the It\^o-SDEs w.r.t.\ $\di Z_t$ with a L\'evy process $Z$. Because of that, we outline the proof, mainly in order to make it visible how the non-linear structure of the jump part effects on the assumptions required.

\begin{proof} Without loss of generality we can assume $p\geq 2$, which will allow us to apply the It\^o formula with the $C^2$-function $|x|^p$.

\smallskip
\noindent
\textbf{1. The first derivative.} Denote $X'_t: = \partial_x X_t$, then
\ba
\di  X'_t&=\tilde a'(X_t)X'_t\, \di t +b'(X_t)X'_t\, \di W_t
+\int_{|z|\leq 1} \Big(\ophi^z_x(X_{t-})-1\Big)X'_{t-}\, \tilde N(\di t, \di z),
\ea
and the It\^o formula yields
% \ba
% |X'_t|^4&= 1+ 4\int_0^t |X'_s|^4 \tilde a'(X_s)\, \di s\\
% &+4\int_0^t |X'_s|^4 b'(X_s)\, \di W_s\\
% &+6\int_0^t |X'_s|^4 b'(X_s)^2\, \di s\\
% &+\int_0^t\int_{|z|\leq  1} \Big( (\ophi^z_x(X_{s-}))^4- 1\Big)|X'_{s-}|^4\, \tilde N(\di t, \di z)\\
% &+\int_0^t\int_{|z|> 1}  \Big((\ophi^z_x(X_{s-}))^4-1\Big)|X'_{s-}|^4\,   N(\di t, \di z)\\
% &+\int_0^t\int_{|z|\leq 1}\Big( (\ophi^z_x(X_s))^4-1-4 ( \ophi^z_x(X_s)-1) \Big)|X'_s|^4\,\nu(\di z)\,\di s\\
% \ea
\ba\label{1-Ito}
|X'_t|^p&= 1+ p\int_0^t |X'_s|^p \tilde a'(X_s)\, \di s
+\frac{p(p-1)}{2}\int_0^t |X'_s|^p b'(X_s)^2\, \di s\\
&+\int_0^t\int_{|z|\leq 1}\Big( |\ophi^z_x(X_s)|^p-1-p( \ophi^z_x(X_s)-1) \Big)|X'_s|^p\,\nu(\di z)\,\di s\\
&+p\int_0^t |X'_s|^p b'(X_s)\, \di W_s+\int_0^t\int_{|z|\leq  1} \Big( |\ophi^z_x(X_{s-})|^p- 1\Big)|X'_{s-}|^p\, \tilde N(\di t, \di z),\\
\ea
% \ba
% (X'_t)^8&= 1+ 8\int_0^t (X'_s)^8 \tilde a'(X_s)\, \di s
% +8\int_0^t (X'_s)^8 b'(X_s)\, \di W_s
% +28\int_0^t (X'_s)^8 b'(X_s)^2\, \di s\\
% &+\int_0^t\int_{|z|\leq  1} \Big( (\ophi^z_x(X_{s-}))^8- 1\Big)(X'_{s-})^8\, \tilde N(\di t, \di z)\\
% &+\int_0^t\int_{|z|\leq 1}\Big( (\ophi^z_x(X_s))^8-1-8 ( \ophi^z_x(X_s)-1) \Big)(X'_s)^8\,\nu(\di z)\,\di s\\
% &+\int_0^t\int_{|z|> 1}  \Big((\ophi^z_x(X_{s-}))^8-1\Big)(X'_{s-})^8\,   N(\di t, \di z)\\
% \ea
where the last two terms are local martingales. Then the standard argument, based on the martingale localization and the Fatou lemma, yields
\ba
\E |X'_t|^p&\leq 1+ p\int_0^t \E |X'_s|^p |\tilde a'(X_s)|\, \di s
+\frac{p(p-1)}{2}\int_0^t \E |X'_s|^p b'(X_s)^2\, \di s\\
&+\int_0^t\int_{|z|\leq 1}\E \Big( |\ophi^z_x(X_s)|^p-1-p( \ophi^z_x(X_s)-1) \Big)|X'_s|^p\,\nu(\di z)\,\di s
\ea
We have the following elementary inequality: for any $p\geq2$ there exists $C_p$ such that for $a, \delta\in \Re$
\ba\label{p-Taylor}
|a+\delta|^p\leq |a|^p+p|a|^{p-1}(\sgn a) \delta+C_p\Big(|a|^{p-2}\delta^2+|\delta|^p\Big).
\ea
In addition, we have $\tilde a', b'$ bounded and, by Lemma \ref{l:phi},
\ba\label{pr1}
|\ophi^z_x(x)-1|\leq C|z|, \quad |z|\leq 1.
\ea
Then, applying \eqref{p-Taylor} with $a=1, \delta=\ophi^z_x(x)-1$ we get from \eqref{1-Ito}
\ba
\E_x|X'_t|^p&\leq 1+ C_{p,T}\int_0^t\E_x |X'_s|^p\, \di s,\quad t\leq T,
\ea
which yields \eqref{p-bound} for $k=1$ by the Gronwall lemma.

\smallskip
\noindent

\textbf{2. The second derivative.} Denote $X''_t: = \partial_{xx} X_t=\partial_x X'_t$, then
\ba
\di  X''_t&=\Big(\tilde a''(X_t)(X'_t)^2 + a'(X_t)X''_t\Big)\, \di t\\
&+\Big(b''(X_t)(X'_t)^2+b'(X_t)X''_t\Big) \, \di W_t\\
&+\int_{|z|\leq 1}\Big[ \ophi^z_{xx}(X_{t-})(X'_{t-})^2 +   \Big(\ophi^z_x(X_{t-})-1\Big)X''_{t-} \Big]   \, \tilde N(\di t, \di z),
\quad X''_0=0,.
\ea By the It\^o formula, localization, and the Fatou lemma,
\ba
\E |X''_t|^p&
\leq p \E \int_0^t \Big(|\tilde a''(X_s)|(X'_s)^2 + |\tilde a'(X_s)||X''_s|  \Big)|X''_s|^{p-1}\, \di s\\
% &+2\E \int_0^t Z_s \int_{|z|\leq 1}\Big[\Big(\ophi^z_{xx}(X_s)-c''(X_s)z  \Big)Y^2_s + \Big( \ophi^z_x(X_s)-1-c'(X_s)z  \Big)Z_s  \Big] \,\nu(\di z)\,\di s\\
&+\frac{p(p-1)}{2} \E \int_0^t \Big(|b''(X_s)||X_s'|^2 +|b'(X_s)||X''_s|\Big)^2 |X''_s|^{p-2}\, \di s\\
% &+\int_0^t \int_{|z|\leq 1}\Big[ (Z_{s-} + \Big[ \ophi^z_{xx}(X_{t-})Y^2_{s-} +   \Big(\ophi^z_x(X_{t-})-1\Big)Z_{s-} \Big]    )^2 -Z_{s-}^2 \Big]  \, \tilde N(\di t, \di z)\\
&+\E \int_0^t \int_{|z|\leq 1}
\Big[ \Big|\ophi^z_{xx}(X_{s})(X'_{s})^2 + \ophi^z_x(X_{s}) X''_{s}  \Big|^p -|X''_{s}|^p
\\
&\hspace*{2cm}- p|X''_{s}|^{p-1} \sgn(X''_s) \Big( \ophi^z_{xx}(X_{s})(X'_{s})^2 +   \Big(\ophi^z_x(X_{s})-1\Big)X''_{s} \Big)   \Big]\, \nu(\di z)\,\di s.\\
\ea
We apply \eqref{p-Taylor} with $a=A(X'')=X'', \delta=\delta(X,X',X'',z)= \ophi^z_{xx}(X)(X')^2 +   \Big(\ophi^z_x(X)-1\Big)X''$.
By Lemma \ref{l:phi}, we have for $|z|\leq 1$
\ba\label{pr2}
|\ophi^z_{xx}(x)|\leq C|z|,
\ea
which together with \eqref{pr1} gives
$$
|\delta(X,X',X'',z)|^2\leq C (|X'|^4+|X''|^2)|z|^2,\quad  |\delta(X,X',X'',z)|^p\leq C (|X'|^{2p}+|X''|^p)|z|^p.
$$
Since $\tilde a', \tilde a'', b', b''$ are bounded and $|z|^p\leq |z|^2$ for $|z|\leq 1,$
this yields inequality
\ba\label{2nd0}
\E|X''_t|^p&\leq C\E \int_0^t \Big(|X''_s|^p + |X''_s|^{p-1}|X'_s|^2 + |X''_s|^{p-2}|X'_s|^4+ |X'_s|^{2p}\Big) \,\di s.
\ea
By the Young inequality
$$
ab\leq \frac{a^{p'}}{p'}+\frac{b^{q'}}{q'}, \quad a,b\geq 0, \quad \frac{1}{p'}+\frac{1}{q'}=1,
$$
we have
$$
|X''_s|^{p-1}|X'_s|^2\leq \frac1p |X'_s|^{2p}+\frac{p-1}{p}|X''_s|^p, \quad |X''_s|^{p-2}|X'_s|^4\leq \frac2p |X'_s|^{2p}+\frac{p-2}{p}|X''_s|^p
$$
Then \eqref{p-bound} with $2p$ and $k=1$, \eqref{2nd0}, and the Gronwall inequality yield \eqref{p-bound} with $p$ and $k=2$.

\textbf{3. The third derivative.} Denote $X'''_t: = \partial_{xxx} X_t=\partial_{xx} X'_t= \partial_x X''_t$, then
\ba
\di  X'''_t&= \Big(\tilde a'''(X_t)(X'_t)^3 + 3\tilde a''(X_t)X'_t X''_t    +  \tilde a'(X_t)X'''_t    \Big)\, \di t\\
&+\Big(b'''(X_t)(X'_t)^3 + 3b''(X_t)X'_t X''_t    +  b'(X_t)X'''_t   \Big) \, \di W_t\\
&+\int_{|z|\leq 1}\Big[ \ophi^z_{xxx}(X_{t-})(X'_{t-})^3 + 3 \ophi^z_{xx}(X_{t-})X'_{t-} X''_{t-} + \Big(\ophi^z_x(X_{t-})-1\Big)X'''_{t-}
\Big]   \, \tilde N(\di t, \di z)\\
\ea
By the It\^o formula, localization, and the Fatou lemma,
\ba
\E |X'''_t|^{p}&\leq  p\E \int_{0}^t  \Big(|\tilde a'''(X_s)||X'_s|^3 + 3|\tilde a''(X_s)||X'_s| |X''_s|    +  |\tilde a'(X_s)||X'''_s|    \Big)|X'''_s|^{p-1}\, \di s\,\\
% &+2\int_0^t  X'''\Big(b'''(X_t)(X'_t)^3 + 3b''(X_t)X'_t X''_t    +  b'(X_t)X'''_t   \Big) \, \di W_s\\
&+\frac{p(p-1)}{2} \E \int_0^t  \Big(|b'''(X_s)||X'_s|^3 + 3|b''(X_s)||X'_s| |X''_s|    +  |b'(X_s)||X'''_s|   \Big)^2 |X'''_s|^{p-2}\, \di s\\
% &\\
% &+\int_{|z|\leq 1}\Big[ \Big(X'''_{s-} +\ophi^z_{xxx}(X_{t-})(X'_{s-})^3 + 3 \ophi^z_{xx}(X)X' X'' + (\ophi^z_x(X)-1)X'''\Big)^2-(X'''_{s-})^2\Big]\, \tilde N(\di t, \di z)\\
% &\\
&+ \E \int_0^t\int_{|z|\leq 1}
\Big[ \Big|X'''_{s}+  \ophi^z_{xxx}(X_{s})(X'_{s})^3 + 3 \ophi^z_{xx}(X_s)X'_s X''_s + \Big(\ophi^z_x(X_s)-1\Big)X'''_s \Big|^{p}-|X'''_{s}|^{p}\\
&-p (X'''_{s})^{p-1}\sgn(X'''_{s}) \Big(  \ophi^z_{xxx}(X_{s})(X'_{s})^3 + 3 \ophi^z_{xx}(X_s)X'_s X''_s + \Big(\ophi^z_x(X_s)-1\Big)X'''_s       \Big)
\Big]     \,   \nu(\di z)\,\di s,
\ea
We apply \eqref{p-Taylor} with $a=X''', \delta=\delta(X,X',X'', X''',z)=  \ophi^z_{xxx}(X)(X')^3 + 3 \ophi^z_{xx}(X)X' X'' + \Big(\ophi^z_x(X)-1\Big)$. We have for  $|z|\leq 1$ by Lemma \ref{l:phi}
\ba\label{pr3}
|\ophi^z_{xxx}(x)|\leq C|z|,
\ea
which together with \eqref{pr1}, \eqref{pr2} and the Young inequality  gives
$$
|\delta(X,X',X'',X''',z)|^2\leq C (|X'|^6+|X''|^4+|X'''|^2)|z|^2,\quad  |\delta(X,X',X'',X''', z)|^p\leq C (|X'|^{3p}+|X''|^{2p}+|X'''|^p)|z|^p.
$$
Since the derivatives of $\tilde a, b$ are bounded and  $|X'||X''|\leq C(|X'|^3+|X''|^{3/2})$ by the Young inequality, we get
\ba
\label{3rd0}
\E|X'''_t|^p&\leq C\E \int_0^t \Big(|X'''_s|^p + |X'''_s|^{p-1}|X'_s|^3+ |X'''_s|^{p-1}|X''_s|^{3/2}
\\& + |X'''_s|^{p-2}|X'_s|^6+ |X'''_s|^{p-2}|X''_s|^{3}+ |X'_s|^{3p}+|X''_s|^{3p/2}\Big) \,\di s.
\ea
Then \eqref{p-bound} for $k=3$ with given $p$ follows from the same bounds with $k=1, 3p$ and $k=2, 3p/2$, the Young inequality, and the Gronwall inequality.

\textbf{4. The fourth derivative.} Denote $X''''_t: = \partial_{xxxx} X_t $, then
\ba
\di  X''''_t&= \Big(a'(X_t) X''''_t + 4\tilde a''(X_t) X'_t X'''_t +6 \tilde a'''(X_t)(X'_t)^2 X''_t  + 3 \tilde a''(X_t) (X''_t)^2+\tilde a''''(X_t)(X'_t)^4 \Big)\, \di t\\
&+\Big( 3 b''(X_t) (X''_t)^2+6 (X'_t)^2 X''_t b'''(X_t)+4 X'_tb''(X_t) X'''_t+(X'_t)^4 b''''(X_t)+b'(X_t) X''''_t   \Big) \, \di W_t\\
&+\int_{|z|\leq 1}\Big[3 (X''_{t-})^2 \ophi^z_{xx}(X_{t-})+4 X'_{t-} \ophi^z_{xx}X'''_{t-}+6 (X'_{t-})^2 X''_{t-}  \ophi^z_{xxx}(X_{t-})+(\ophi^z_x(X_{t-})-1) X''''_{t-}
\\&+(X'_{t-})^4 \ophi^z_{xxxx}(X_{t-})
\Big]   \, \tilde N(\di t, \di z).
\ea
By the It\^o formula, localization, and the Fatou lemma,
\ba
\E |X''''_t|^{p}&\leq  p\E  \int_0^t \Big(|a'(X_s)| |X''''_s| + 4|\tilde a''(X_s)| |X'_s| |X'''_s|
\\&\hspace*{1cm}+6 |\tilde a'''(X_s)||X'_s|^2 |X''_s|  + 3 |\tilde a''(X_s)| |X''_s|^2+|\tilde a''''(X_s)|X'_s|^4 \Big)|X''''_s|^{p-1}\, \di s\\
&+\frac{p(p-1)}{2} \E \int_0^t \Big( 3 |b''(X_s)| |X''_s|^2+6 |X'_s|^2 |X''_s| |b'''(X_s)|+4 |X'_s|b''(X_s)| |X'''_s|
\\&\hspace*{1cm}+|X'_s|^4 |b''''(X_s)|+|b'(X_s)| |X''''_s|   \Big)^2 |X''''_t|^{p-2}\, \di s\\
&+ \E \int_{|z|\leq  1}  \Big[ \Big(3 (X''_s)^2 \ophi^z_{xx}(X_s)+4 X'_s \ophi^z_{xx}(X_s)X'''_s
+6 (X'_s)^2 X''_s  \ophi^z_{xxx}(X_s)+\ophi^z_x(X_s) X''''_s\\&+(X'_s)^4 \ophi^z_{xxxx}(X_s)\Big)^2
- (X''''_s)^2-p\Big(3 (X''_s)^2 \ophi^z_{xx}(X_s)+4 X'_s \ophi^z_{xx}X'''_s+6 (X'_s)^2 X''_s  \ophi^z_{xxx}(X_s)
\\&+(\ophi^z_x(X_s)-1) X''''+(X'_s)^4 \ophi^z_{xxxx}(X_s)
\Big) (X''''_s)^{p-1}\sgn(X''''_s) \Big]     \, \nu(\di z) \di s.
\ea
We apply \eqref{p-Taylor} with $a=X'''',$
$$
\delta=\delta(X,X',X'', X''', X'''',z)= 3 (X'')^2 \ophi^z_{xx}(X)+4 X' \ophi^z_{xx}X'''+6 (X')^2 X''  \ophi^z_{xxx}(X)
+(\ophi^z_x(X)-1) X''''+(X')^4 \ophi^z_{xxxx}(X).
$$
We have for  $|z|\leq 1$ by Lemma \ref{l:phi}
\ba\label{pr4}
|\ophi^z_{xxxz}(x)|\leq C|z|,
\ea
which together with \eqref{pr1}, \eqref{pr2}, \eqref{pr3} and the Young inequality  gives
\ba
&|\delta(X,X',X'',X''',z)|^2\leq C (|X'|^8+|X''|^4+|X'''|^{8/3}+|X''''|^2)|z|^2, \\
&|\delta(X,X',X'',X''',z)|^p\leq C (|X'|^{4p}+|X''|^{2p}+|X'''|^{4p/3}+|X''''|^p)|z|^p.
\ea
Since the derivatives of $\tilde a, b$ are bounded, applying the Young inequality once again we get
\ba\label{e:3p}
\E|X'''_t|^p&\leq C\E \int_0^t \Big(|X''''_s|^p + |X''''_s|^{p-1}|X'_s|^4+ |X''''_s|^{p-1}|X''_s|^{2}+ |X''''_s|^{p-1}|X'''_s|^{4/3}
\\& + |X''''_s|^{p-2}|X'_s|^8+ |X''''_s|^{p-2}|X''_s|^{4}+ |X''''_s|^{p-2}|X'''_s|^{8/3}+ |X'_s|^{4p}+|X''_s|^{2p}+|X'''_s|^{4p/3}\Big) \,\di s.
\ea
Then \eqref{p-bound} for $k=4$ with given $p$ follows from the Young inequality, the Gronwall inequality, and the bounds \eqref{p-bound} with $k=1,2,3$ and $p'$ equal $4p, 2p, 4p/3$, respectively.
\end{proof}

Now, let us consider the general case of non-trivial large jump part.
The semigroup $P_t$ of the solution to \eqref{e:SDEM} admits the following representation. Consider the SDE  \eqref{e:SDEtrun}, which corresponds to the driving noise with large jumps (i.e. $|z|>1$) truncated away. Denote the corresponding semigroup $\tilde P_t, t\geq 0$. Denote by $\mathcal Q$ the operator which corresponds to a single large jump of the driving noise:
\ba\label{Q}
\mathcal Q f(x)=\int_{|z|>1}\Big(f(\ophi^z(x))-f(x)\Big)\,\nu(\di z).
\ea
Then we have
\ba
P_t=\ex^{-\lambda t}\tilde P_t+\sum_{k=1}^\infty \ex^{-\lambda t}\int_{0\leq s_1\leq \dots\leq s_k\leq t}
\tilde P_{t-s_k}\mathcal Q\tilde P_{s_k-s_{k-1}}\mathcal Q \dots \mathcal Q P_{s_1}\, \di s_1\dots\di s_k,
\ea
where $\lambda=\nu(|z|>1)$ is the intensity of large jumps. The above representation follows easily by independence of the processes
\ba
\tilde Z_t=\int_0^t\int_{|z|\leq 1} z\, \tilde N(\di s, \di z), \quad \text{and}\quad Z_t-\tilde Z_t=\int_0^t\int_{|z|>1} z\, N(\di s, \di z)
\ea
and the compound Poisson structure of $Z-\tilde Z$.

We have shown in the first part of the proof that
\ba
\|\tilde P_t\|_{C^4\to C^4}\leq C_T, \quad t\leq T.
\ea
On the other hand, for the function $\mathcal Q f$ given by the integral formula \eqref{Q} its derivatives of the orders $1,\dots 4$ admit integral representations similar to \eqref{e:f1}--\eqref{f4}, and then it is a direct calculation to see that
\ba
\|\mathcal Q\|_{C^4\to C^4}\leq C_{\mathcal Q}.
\ea
Then for the semigroup $P_t$ we have for $t\leq T$
\ba
\|P_t\|_{C^4\to C^4}\leq \ex^{-\lambda t}C_T
+\sum_{k=1}^\infty\ex^{-\lambda t}\frac{t^k}{k!}(C_T)^{k+1}(C_{\mathcal Q})^k=C_T\ex^{t(C_{\mathcal Q}C_T-\lambda)}
\leq C_T\ex^{T(C_{\mathcal Q}C_T-\lambda)_+},
\ea
which completes the proof.

\appendix

\section{Properties of $\ophi^z(u;x)$ and its derivatives\label{a:ophi}}

\begin{lem}
\label{l:phi}
Let \emph{\textbf{H}}$_{a,b,c}$ holds true
and let
\ba
\phi(u;x,z)=\ophi^z(u;x)-x-c(x)zu,\quad u\in[0,1].
\ea
Then there is a constant $C>0$ such that for all $|z|\leq 1$ and all $x\in \bR$
\ba
|\phi(u;x,z)|&\leq  C\cdot z^2\cdot|c(x)|,\\
|\nabla^k_x \phi(u;x,z)|&\leq C\cdot z^2,\quad 1\leq k\leq 4.
\ea
In particular, the effective drift $\tilde a\in C^4(\bR,\bR)$ and $\|\nabla^k \tilde a\|<\infty$, $k=1,\dots,4$,
and for $|z|\leq 1$
\ba
&|\ophi^z(u;x)-x|\leq C(1+|x|),\\
&|\ophi^z_{x}(u;x)-1|\leq C|z|,\\
&|\nabla_x^k \ophi^z(u;x)|\leq C|z|,\quad k=2,3,4.
\ea
\end{lem}
\begin{proof}

Estimate the integral term.

1. We write
\ba
\ophi^z(u;x)=x+c(x)zu+\phi(u;x,z),\quad u\in[0,1]
\ea
Then
\ba
\frac{\di}{\di u}\ophi^z(u;x)=c(x)z+ \dot \phi(u;x,z)&=c(x+c(x)zu+\phi(u;x,z))z\\
&=c(x)z + c'(\xi)\Big(c(x)zu+\phi(u;x,z)\Big)z,\quad \xi=\xi(u,x,z)
\ea
Hence
\ba
\dot \phi(u;x,z)&= c'(\xi)c(x)z^2u+   \phi(u;x,z)c'(\xi)z,\\
|\phi(u;x,z)|&\leq \int_0^u \Big( \|c'\||c(x)|z^2 + \|c'\||z||\phi(r;x,z)|\Big)\,\di r,\\
|\phi(u;x,z)|&\leq z^2\|c'\||c(x)|\cdot \ex^{\|c'\|}.
\ea
Hence
\ba
|\ophi^z(x)-x-c(x)z |\leq  z^2\|c'\||c(x)|\cdot \ex^{\|c'\|}
\ea
and $\tilde a$ is of linear growth.

2. Analogously,
\ba
\ophi^z_x(u;x)=1+c'(x)zu+\phi_x(u;x,z),\quad u\in[0,1]
\ea
Then
\ba
\frac{\di}{\di u}\ophi^z_x(u;x)=c'(x)z + \dot \phi_x(u;x,z)&=c'(\ophi^z(u;x))\ophi^z_x(u;x) z\\
&=c'(\ophi^z(u;x))\Big(1+c'(x)zu+\phi_x(u;x,z)\Big)z
\ea
Hence
\ba
\dot\phi_x(u;x,z)&=\Big(c'(\ophi^z(u;x))-c'(x) \Big)z + c'(\ophi^z(u;x))c'(x)z^2 u +c'(\ophi^z(u;x))\phi_x(u;x,z)z\\
&=z^2 \int_0^u c''(\ophi^z(r;x))c(\ophi^z(r;x))\,\di r + c'(\ophi^z(u;x))c'(x)z^2 u +c'(\ophi^z(u;x))\phi_x(u;x,z)z
\ea
Hence
\ba
|\phi_x(u;x,z)|&\leq( \|c''c\|+\|c'\|^2)z^2 +  \|c'\||z| \int_0^u|\phi(r;x,z)|\,\di r,\\
|\phi_x(u;x,z)|&\leq z^2 ( \|c''c\|+\|c'\|^2)  \cdot \ex^{\|c'\|}.
\ea

3. Analogously,
\ba
\ophi^z_{xx}(u;x)=c''(x)z u+\phi_{xx}(u;x,z),\quad u\in[0,1]
\ea
Then
\ba
\frac{\di}{\di u}\ophi^z_{xx}(u;x)&=c''(x)z + \dot \phi_{xx}(u;x,z)\\
&=c''(\ophi^z(u;x))\Big(1+ c'(x)zu+\phi^z_x(u;x)\Big)^2 z+c'(\ophi^z(u;x))\Big(c''(x)zu+\phi_{xx}(u;x,z)\Big)z\\
&=c''(\ophi^z(u;x))\Big(1+  2(c'(x)zu+\phi^z_x(u;x)) + (c'(x)zu+\phi^z_x(u;x))^2  \Big) z\\
&+c'(\ophi^z(u;x))\Big(c''(x)zu+\phi_{xx}(u;x,z)\Big)z
\ea
Taking into account that $\|c'''c\|<\infty$ and
\ba
c''(\ophi^z(u;x))-c''(x)= z \int_0^u c'''(\ophi^z(r;x))c(\ophi^z(r;x))\,\di r
\ea
we get that
\ba
|\phi_x(u;x,z)|&\leq z^2 \cdot C_2  \cdot \ex^{\|c'\|}.
\ea
4. The higher derivatives are checked analogously.
\end{proof}

We have the following formulae for the derivatives of the Marcus flow $x\mapsto \ophi^z(x)$.
These derivatives are hence solutions of
non-autonomous non-homogeneous linear differential equations.

\ba
\label{e:eqophi}
&\frac{\di}{\di u}\ophi^z_{x}=z c'(\ophi^z) \ophi^z_{x},\quad \ophi^z(0;x)=1,\\
% \ophi^z_{x}(u)&=  \ex^{\int_0^t c'(\ophi^z)\,\di r}\\
%
&\frac{\di}{\di u}\ophi^z_{xx}=z c''(\ophi^z) \ophi_x^2+z c'(\ophi^z) \ophi^z_{xx},\\
% \ophi^z_{xx}(u)&=  \ex^{\int_0^t c'(\ophi^z)\,\di r}\,\di s  \cdot  \int_0^1 z c''(\ophi^z) \ophi_x^2\cdot \ex^{-\int_0^s c'(\ophi^z)\,\di r}\,\di s
% =\ex^{\int_0^t c'(\ophi^z)\,\di r}\,\di s  \cdot  \int_0^1 z c''(\ophi^z) \ophi_x(u;x)\,\di s
% ,\\
%
&\frac{\di}{\di u}\ophi^z_{xxx}=z \Big(c'''(\ophi^z) \ophi_x^3+3 c''(\ophi^z)\ophi^z_x \ophi_{xx}^z\Big)+z c'(\ophi^z) \ophi^z_{xxx},\\
% \ophi^z_{xxx}(u)&=  \ex^{\int_0^t c'(\ophi^z)\,\di r}\,\di s  \cdot  \int_0^1 z \Big(c'''(\ophi^z) \ophi_x^3
% +3 c''(\ophi^z)\ophi^z_x \ophi_{xx}^z\Big)\ex^{-\int_0^s c'(\ophi^z)\,\di r}\,\di s\\
% &=\ex^{\int_0^t c'(\ophi^z)\,\di r}\,\di s  \cdot  \int_0^1 z \Big(c'''(\ophi^z) \ophi_x^2
% +3 c''(\ophi^z)\ophi_{xx}^z\Big)\,\di s,\\
%
&\frac{\di}{\di u}\ophi^z_{xxxx}=z \Big(c''''(\ophi^z) (\ophi_x^z)^4+ 6 c'''(\ophi^z)\ophi^z_x (\ophi_{xx}^z)^2
+ 3 c''(\ophi^z)(\ophi_{xx}^z)^2
+ 4c'(\ophi^z)\ophi_{x}^z(\ophi_{xxx}^z)^3 \Big)
+z c'(\ophi^z) \ophi^z_{xxxx},\\
\ea
\begin{lem}
\label{l:Aa}
Under assumption \emph{\textbf{H}}$_{a,b,c}$ we have for all $|z|>1$ and $x\in\bR$
\ba
|\ophi^z_{x}(u;x)|& \leq \ex^{\|c'\||z|},\quad \\
|\ophi^z_{xx}(u;x)| &\leq |z|\ex^{3\|c'\||z|},\quad\\
|\ophi^z_{xxx}(u;x)| &\leq |z|^2\ex^{5\|c'\||z|},\quad\\
|\ophi^z_{xxxx}(u;x)| &\leq |z|^3\ex^{8\|c'\||z|},\quad u\in[0,1].
\ea
In particular,
\ba
|\ophi^z(x)-x|& \leq |x|(1+\ex^{\|c'\||z|}),\\
|\ophi^z_x(x)-1|& \leq  1+\ex^{\|c'\||z|}.
\ea
\end{lem}
\begin{proof}
Indeed, solving the linear equations \eqref{e:eqophi} we get
\ba
\label{e:grad}
\ophi^z_{x}(u)&=  \ex^{\int_0^t c'(\ophi^z)z\,\di r},\\
\ophi^z_{xx}(u)&=  \int_0^u z c''(\ophi^z) \ophi_x^2\cdot \ex^{\int_s^u c'(\ophi^z)z\,\di r}\,\di s,\\
% =\int_0^u z c''(\ophi^z)  \ex^{2\int_0^s c'(\ophi^z)z\,\di r}   \cdot \ex^{\int_s^u c'(\ophi^z)z\,\di r}\,\di s
% \\
% &= \ex^{\int_0^u c'(\ophi^z)z\,\di r}\cdot \int_0^u z c''(\ophi^z)  \ex^{\int_0^s c'(\ophi^z)z\,\di r}  \,\di s
% =\ophi^z_{x}(u)\cdot  \int_0^u z c''(\ophi^z)  \cdot \ophi^z_x(s) \,\di s,
\ophi^z_{xxx}(u)&=  \int_0^u z \Big(c'''(\ophi^z) \ophi_x^3
+3 c''(\ophi^z)\ophi^z_x \ophi_{xx}^z\Big)\ex^{\int_s^u c'(\ophi^z)z\,\di r}\,\di s,
% &=\ophi^z_{x}(u)\cdot  \int_0^u z \Big(c'''(\ophi^z) \ophi_x^2
% +3 c''(\ophi^z)\ophi_{xx}^z\Big)\,\di s
\\
\ophi^z_{xxxx}(u)&=  \int_0^u z \Big(
c''''(\ophi^z) (\ophi_x^z)^4+ 6 c'''(\ophi^z)\ophi^z_x (\ophi_{xx}^z)^2
+ 3 c''(\ophi^z)(\ophi_{xx}^z)^2
+ 4c''(\ophi^z)\ophi_{x}^z\ophi_{xxx}^z
\Big)\ex^{\int_s^u c'(\ophi^z)z\,\di r}\,\di s.\\
% &=\ophi_x^z(u)\cdot \int_0^u z \Big(
% c''''(\ophi^z) (\ophi_x^z)^3+ 6 c'''(\ophi^z)\cdot (\ophi_{xx}^z)^2
% + 4c''(\ophi^z)\ophi_{xxx}^z
% \Big)\,\di s
% +
% 3\int_0^u z c''(\ophi^z)(\ophi_{xx}^z)^2\cdot \ex^{\int_s^u c'(\ophi^z)z\,\di r}\,\di s,
\ea
and hence the estimates follow.

By the Gronwall lemma, $|\ophi^z(x)|\leq |x|\ex^{\|c'\|\cdot |z|}$, and
\ba
|\ophi^z(x)-x|\leq 1+\ex^{\|c'\|\cdot |z|}.
\ea
\end{proof}

In the multidimensional setting, solutions should be written in terms of the fundamental solution
of the linear differential equation with the matrix $Dc(\ophi^z(u;x))z$ and the estimates
\eqref{e:grad} follow, for example from \cite[Section IV.4]{Hartman-64}).

\section{Properties of $\psi(u;x;\tau,w,z)$ and its derivatives\label{a:psi}}

% \subsection{Gronwall type estimates for $\psi$.}

For the estimates of the Lemma \ref{l:QQ} we need the following elementary inequalities.

\begin{lem}
\label{l:psi}
Let \emph{\textbf{H}}$_{a,b,c}$ holds true. Then
there is a constant $C>0$ such that for
all $\tau\geq 0$, $w\in\bR$, $z\in\bR$, and $x\in \bR$
\begin{align}
\label{e:psi0}
\sup_{u\in [0,1]}|\psi(u;x;\tau,w,z)|&\leq C(1+|x|)\cdot \ex^{\|a'\|\tau+ \|b'\||w|+ \|c'\||z|},\\
\label{e:psi1}
\sup_{u\in [0,1]}|\partial_i \psi(u;x,\tau,w,z)|&\leq C  (1+|x|)\cdot  \ex^{2(\|a'\|\tau+ \|b'\||w|+ \|c'\||z|)},\quad i\in\{\tau,w,z\} ,\\
\label{e:psi2}
\sup_{u\in [0,1]}|\partial_{ij} \psi(u;x;\tau,w,z)|&\leq C (1+ x^2)\cdot (1+\tau+|w|+|z|)\ex^{5(\|a'\|\tau+ \|b'\||w|+ \|c'\||z|) },
\quad i,j\in\{\tau,w,z\} ,\\
\label{e:psi3}
\sup_{u\in [0,1]}|\psi_{ijk}(u;x;\tau,w,z)|&\leq C  (1+|x|^3)\cdot (1+\tau+|z|+|w|)^2 \ex^{8  (\|a'\|\tau+ \|b'\||w|+ \|c'\||z|)  } ,
\quad i,j,k\in\{\tau,w,z\} ,\\
\label{e:psi4}
\sup_{u\in [0,1]}|\psi_{ijkl}(u;x;\tau,w,z)|&\leq C  (1+|x|^4)\cdot (1+\tau +|z|+|w|)^3\ex^{11  (\|a'\|\tau+ \|b'\||w|+ \|c'\||z|)  } ,
\quad i,j,k,l\in\{\tau,w,z\} .
\end{align}
\end{lem}
\begin{proof}
These estimates are obtained directly.

\noindent
\textbf{0. Estimate of $\psi$.}
For $\tau,w,z\in\bR$, denote $\psi(u)=\psi(u;x;\tau,w,z)$ the solution to the Cauchy problem
\ba
&\frac{\di}{\di u} \psi(u)=a(\psi(u))\tau+b(\psi(u))w+c(\psi(u))z, \\
&\psi(0)=x,\quad u\in[0,1].
\ea
Since
\ba
 |a(x)|\leq |a(0)|+\|a'\||x|,\quad  |b(x)|\leq |b(0)|+\|b'\||x|,\quad |c(x)|\leq |c(0)|+\|c'\| |x|,
\ea
the Gronwall inequality yields \eqref{e:psi0} for some $C>0$.

\noindent
\textbf{1. Estimates of $\psi_\tau$, $\psi_w$, $\psi_z$.}
The derivative w.r.t.\ $\tau$ satisfies the lienar non-autonomous ODE
\ba
\frac{\di}{\di u}&\psi_\tau=a(\psi)+(a'(\psi)\tau +b'(\psi)w+c'(\psi)z) \psi_\tau\\
&\psi_\tau(0;x;\tau,w,z)=0\\
\ea
which can be solved explicitly
\ba
\psi_\tau(u)&=\int_0^u a(\psi(s))\ex^{\int_{s}^u  (\tau a'(\psi(r))+w b'(\psi(r))+z c'(\psi(r)))\,\di r  }\, \di s,
\ea
Applying the estimate \eqref{e:psi0} we get (for a different constant $C>0$)
\ba
\sup_{u\in [0,1]}|\psi_\tau(u;x;\tau,w,z)|&\leq  C(1+|x|)\cdot \ex^{2(\|a'\||\tau|+\|b'\||w|+\|c'\||z|)} .
\ea
Due to the symmetry of the ODE for $\psi$ w.r.t.\ $\tau$, $w$, and $z$ the same estimate holds for $\psi_w$ and $\psi_z$.

\noindent
\textbf{2. Estimates of  $\psi_{\tau\tau}$, $\psi_{\tau w}$, $\psi_{\tau z}$, $\psi_{ww}$, $\psi_{wz}$, $\psi_{zz}$.}
We consider derivatives $\psi_{\tau\tau}$ and $\psi_{\tau w}$,
\ba
\frac{\di}{\di u}&\psi_{\tau\tau}
=2a'(\psi) \psi_\tau  + \Big(a''(\psi)\tau  +b''(\psi)w +c''(\psi)z\Big)\psi_\tau^2
+ \Big(a'(\psi)\tau  + b'(\psi)w + c'(\psi)z\Big) \psi_{\tau\tau},\\
&\psi_{\tau\tau}(0;x;t,w,z)=0,\\
\frac{\di}{\di u}&\psi_{\tau w}=a'(\psi)\psi_w+ b'(\psi)\psi_\tau+
\Big(a''(\psi)\tau +b''(\psi)w + c''(\psi)z\Big) \psi_\tau\cdot \psi_w
+\Big(a'(\psi)\tau +b'(\psi)w+c'(\psi)z\Big) \psi_{\tau w},
\\
&\psi_{\tau w}(0;x;\tau,w,z)=0.
\ea
Writing down the solution explicitly and using the estimates from the previous steps yields the result.

\noindent
\textbf{3. Estimates of $\psi_{\tau\tau\tau}$, $\psi_{\tau \tau w}$, $\psi_{\tau \tau  z}$, $\psi_{\tau ww}$, \dots}
We consider derivatives $\psi_{\tau\tau\tau}$ and $\psi_{\tau \tau w}$, and $\psi_{\tau w z}$
\ba
\frac{\di}{\di u}&\psi_{\tau\tau\tau}
=3a''(\psi) \psi_\tau^2 + 3a'(\psi) \psi_{\tau\tau} + 3\Big(a''(\psi)\tau  +b''(\psi)w +c''(\psi)z\Big)\psi_\tau\psi_{\tau\tau}\\
&+ \Big(a'''(\psi)\tau     +b'''(\psi)w  +c'''(\psi)z \Big)\psi_\tau^3
%
% &+ \Big(a''(\psi)\tau    + b''(\psi)w  + c''(\psi)z \Big) \psi_\tau\psi_{\tau\tau}\\
+ \Big(a'(\psi)\tau  + b'(\psi)w + c'(\psi)z\Big) \psi_{\tau\tau\tau},\\
&\psi_{\tau\tau\tau}(0;x;t,w,z)=0,\\
\ea
\ba
\frac{\di}{\di u}\psi_{\tau\tau w}
&= b''(\psi) \psi_\tau^2 +  b'(\psi)\psi_{\tau\tau}  +  2a'(\psi) \psi_{\tau w}
+ 2\Big(a''(\psi)(1+\tau)  +b''(\psi)w +c''(\psi)z\Big)\psi_\tau\psi_w\\
&+ \Big(a'''(\psi)\tau   +b'''(\psi)w +c'''(\psi)z  \Big)\psi_\tau^2\psi_w
+ \Big(a''(\psi)\tau + b''(\psi)w  + c''(\psi)z \Big) \psi_{\tau\tau}\psi_w \\
&+ \Big(a'(\psi)\tau  + b'(\psi)w + c'(\psi)z\Big) \psi_{\tau\tau w},\\
\ea
\ba
\frac{\di}{\di u}\psi_{\tau w z}&=c''(\psi) \psi_\tau\psi_w+ b''(\psi)\psi_\tau\psi_z +  a''(\psi)\psi_w\psi_z
+ c'(\psi) \psi_{\tau w}+ b'(\psi)\psi_{\tau z}+ a'(\psi)\psi_{wz}\\
&+\Big(a''(\psi)\tau +b''(\psi)w + c''(\psi)z\Big) \Big(\psi_\tau \psi_{wz} + \psi_{\tau w}\psi_z +\psi_{\tau z} \psi_w  \Big) \\
% &+\Big(a''(\psi)\tau +b''(\psi)w + c''(\psi)z\Big) \\
% &+\Big(a''(\psi)\tau +b''(\psi)w +  c''(\psi)z\Big) + \\&\\
&+\Big(a'''(\psi)\tau  +b'''(\psi)w  + c'''(\psi)z\Big) \psi_\tau \psi_w\psi_z
+\Big(a'(\psi)\tau +b'(\psi)w+c'(\psi)z\Big) \psi_{\tau w z}\\
&\psi_{\tau wz}(0;x;\tau,w,z)=0.
\ea

\noindent
\textbf{4. Estimates of $\psi_{\tau\tau\tau\tau}$, $\psi_{\tau \tau \tau w}$, $\psi_{\tau \tau \tau z}$, $\psi_{\tau \tau ww}$, \dots}

We consider derivatives $\psi_{\tau\tau\tau\tau}$ and $\psi_{\tau \tau \tau w}$, and $\psi_{\tau \tau w w}$, and $\psi_{\tau \tau w z}$:
\ba
\frac{\di}{\di u}\psi_{\tau \tau\tau\tau}&=
\Big(\tau  a^{(4)}(\psi) +w b^{(4)}(\psi) +z c^{(4)}(\psi)\Big) \psi_\tau^4
+6 \Big(\tau  a'''(\psi) +w b'''(\psi)+ z c'''(\psi)\Big)   \psi_\tau^2 \psi_{\tau\tau}\\
&+\Big(\tau  a''(\psi)+ w b''(\psi)+ z c''(\psi)\Big)\Big(3\psi_{\tau\tau}^2+4 \psi_\tau\psi_{\tau\tau\tau}\Big)\\
&+4 \Big(a'''(\psi)\psi_\tau^3+3 a''(\psi)\psi_\tau\psi_{\tau\tau}+a'(\psi)\psi_{\tau\tau\tau}\Big)
+\Big(\tau  a'(\psi)+w b'(\psi)+z c'(\psi)\Big) \psi_{\tau \tau\tau\tau}
\ea

\ba
\frac{\di}{\di u}\psi_{\tau \tau\tau w}
&=b'''(\psi)\psi_\tau^3 +b'(\psi)\psi_{\tau\tau\tau} +3 b''(\psi)\psi_\tau\psi_{\tau\tau} \\
&+3 \Big(a'''(\psi)\psi_w\psi_\tau^2+2 a''(\psi)\psi_\tau\psi_{\tau w}+a''(\psi)\psi_w\psi_{\tau\tau}+a'(\psi)\psi_{\tau\tau w}\Big)\\
&+\Big(\tau  a^{(4)}(\psi)+w b^{(4)}(\psi)+z c^{(4)}(\psi)\Big)\psi_\tau^3\psi_w\\
&+3 \Big(\tau  a'''(\psi)+ w b'''(\psi) +z c'''(\psi)\Big)\Big(\psi_\tau^2\psi_{\tau w}+ \psi_w\psi_\tau\psi_{\tau\tau}\Big)\\
&+\Big(\tau  a''(\psi)+ w b''(\psi)+ z c''(\psi)\Big)\Big(3\psi_{\tau w}\psi_{\tau\tau}+3\psi_\tau\psi_{\tau\tau w}+ \psi_w\psi_{\tau\tau\tau}\Big)\\
&+\Big(\tau  a'\psi+w b'\psi+z c'\psi\Big) \psi_{\tau\tau\tau w}
\ea

\ba
\frac{\di}{\di u}\psi_{\tau \tau w w}
&=\tau  a^{(4)}(\psi) \psi_\tau^2\psi_w^2 +a'''(\psi)\psi_\tau \psi_w^2+4 a''(\psi) \psi_w \psi_{\tau w}+2a''(\psi)\psi_\tau \psi_{ww}+2 a'(\psi) \psi_{\tau ww}\\
&+2 b'''(\psi)\psi_\tau^2 \psi_w +4 b''(\psi) \psi_\tau \psi_{\tau w}+2 b''(\psi) \psi_w \psi_{\tau \tau}+2 b'(\psi) \psi_{\tau \tau w} \\
&+  \Big(\tau  a^{(4)}(\psi) +w b^{(4)}(\psi) + z c^{(4)}(\psi)   \Big) \psi_\tau^2 \psi_w^2   \\
&+\Big(\tau  a'''(\psi)+ w b'''(\psi) + z c'''(\psi)\Big)\Big(4 \psi_w \psi_\tau \psi_{\tau w}+\psi_\tau^2\psi_{ww}+ \psi_{\tau\tau}\psi_{w}^2\Big)\\
&+\Big(\tau  a''(\psi) +w b''(\psi) +z c''(\psi)\Big) \Big(2 \psi_{\tau w}^2+2 \psi_\tau \psi_{\tau ww}+2\psi_w \psi_{\tau \tau w}+ \psi_{\tau \tau}\psi_{ww}  \Big)\\
&+\Big(\tau  a'(\psi)+w b'(\psi)+z c'(\psi)\Big) \psi_{\tau\tau ww}
\ea

\ba
\frac{\di}{\di u}\psi_{\tau \tau w w}
&=  2 a'(\psi) \psi_{\tau w z} + b'(\psi) \psi_{\tau\tau z}+c'(\psi) \psi_{\tau \tau w}\\
&+2 a''(\psi) \Big(\psi_{\tau z}\psi_w +\psi_{\tau w}\psi_z +\psi_\tau\psi_{wz} \Big)
+b''(\psi)\Big( 2\psi_\tau \psi_{\tau z}+\psi_{\tau\tau} \psi_z\Big) +c''(\psi) \Big(2\psi_\tau \psi_{\tau w}+ \psi_{\tau\tau}\psi_w\Big)\\
&+2 a'''(\psi) \psi_\tau\psi_w\psi_z  +b'''(\psi)\psi_\tau^2 \psi_z +c'''(\psi)\psi_\tau^2 \psi_w \\
&+\Big(\tau  a^{(4)}(\psi) +w b^{(4)}(\psi) +z c^{(4)}(\psi)\Big) \psi_\tau^2\psi_w \psi_z \\
&+\Big(\tau  a'''(\psi) +w b'''(\psi) +z c'''(\psi)\Big)\Big( \psi_\tau^2\psi_{wz} + 2\psi_w \psi_\tau \psi_{\tau z}+2 \psi_\tau \psi_{\tau w} \psi_z + \psi_{\tau\tau}\psi_w \psi_z  \Big)\\
&+\Big(\tau  a''(\psi) +w b''(\psi)+ z c''(\psi)\Big)
\Big( 2\psi_{\tau w}\psi_{\tau z} + 2\psi_\tau \psi_{\tau w z}+ \psi_{\tau\tau}\psi_{wz}+\psi_{\tau\tau z}\psi_w +\psi_{\tau \tau w} \psi_z \Big) \\
&+\Big(\tau  a'(\psi)+w b'(\psi)+z c'(\psi)\Big) \psi_{\tau\tau wz}.
\ea

\end{proof}

\end{document}